\documentclass{article}

\usepackage[latin1]{inputenc}
\usepackage{subfigure}
\usepackage{amsmath}
\usepackage{enumerate}
\usepackage{amssymb}
\usepackage[dvipsnames]{xcolor}
\usepackage[normalem]{ulem}
\usepackage{paralist}
\usepackage{latexsym}
\usepackage{tikz}
\usepackage{mathrsfs}
\usepackage{amsthm}
\usepackage{thm-restate}
\usepackage{hyperref}

\newcommand{\ol}[1]{\ensuremath{\overline{#1}}} 
\newcommand{\del}{\Delta}
\newcommand{\nab}{\nabla}

\newcommand{\conv}{\operatorname{Conv}}
\newcommand{\cork}{\operatorname{cork}}
\newcommand{\nul}{\operatorname{null}}
\newcommand{\tri}{\operatorname{\triangle}} 

\newtheorem{thm}{Theorem}[section]

\newtheorem{lemma}[thm]{Lemma}
\newtheorem{sublemma}[thm]{Claim}

\newtheorem{cor}[thm]{Corollary}
\newtheorem{prop}[thm]{Proposition}
\newtheorem{dfn}[thm]{Definition}
\theoremstyle{remark}
\newtheorem{remark}[thm]{Remark}
\newtheorem{question}[thm]{Question}
\newtheorem{Example}[thm]{Example}
\newtheorem{algorithm}[thm]{Algorithm}
\newenvironment{ex}{\begin{Example}\pushQED{\qed}}{\popQED\end{Example}}

\theoremstyle{plain}

\newtheorem{citedfn}[thm]{Definition}

\newenvironment{subproof}[1][\proofname]{%
  \begin{proof}[#1]%
}{%
  \end{proof}%
}

\newtheorem*{thmcoeffs}{Theorem \ref{coeffs}}

\numberwithin{equation}{thm}

\author{Amanda Cameron\thanks{Email: \texttt{cameron\char64mis.mpg.de}}%
  \and Alex Fink\thanks{Email: \texttt{a.fink\char64qmul.ac.uk}}}
\title{The Tutte polynomial via lattice point counting}

\date{\today}

\begin{document}

\maketitle
\begin{abstract}
We recover the Tutte polynomial of a matroid, up to change of coordinates, 
from an Ehrhart-style polynomial
counting lattice points in the Minkowski sum of its base polytope and scalings of simplices.
Our polynomial has coefficients of alternating sign
with a combinatorial interpretation closely tied to the Dawson partition.
Our definition extends in a straightforward way to polymatroids,
and in this setting our polynomial has K\'alm\'an's internal and external activity polynomials as its univariate specialisations.
\end{abstract}

\section{Introduction}
\label{sec:in}

The Tutte polynomial $T_M(x,y)$ of a matroid $M$, 
formulated by Tutte for graphs and generalised to matroids by Crapo,
is perhaps the invariant most studied by researchers in either field,
on account of its diverse applications.
Most straightforwardly, any numerical function of matroids or graphs (or function valued in another ring) which can be computed by a deletion-contraction recurrence
can also be obtained as an evaluation of the Tutte polynomial. 
Examples include the number of bases and the number of independent sets in a given matroid, and the number of acyclic orientations of a graph.

A matroid is expediently encoded by its \emph{matroid (base) polytope}, 
the convex hull of the indicator vectors of the bases.
One useful feature of this polytope is that no other \emph{lattice points}, i.e.\ points with integer coordinates, are caught in the convex hull.
Therefore, counting the lattice points in the base polytope of $M$ tells us the number of bases of $M$,
which equals $T_M(1,1)$. 

The Tutte polynomial can be written as a generating function for bases of a matroid according to their \emph{internal} and \emph{external activity},
two statistics that count elements which permit no basis exchanges of certain shapes.
K\'alm\'an \cite{kalman} observed that the active elements of a basis can be readily discerned 
from the matroid polytope:
given a lattice point of this polytope, one inspects in which ways one can increment one coordinate, decrement another, and remain in the polytope.
The lattice points one reaches by incrementing a coordinate are those contained in the Minkowski sum with the standard simplex $\conv\{e_1,\ldots,e_n\}$,
and decrementing a coordinate is similarly encoded by the upside-down simplex $\conv\{-e_1,\ldots,-e_n\}$.

Our main result, Theorem~\ref{lpg}, proves that the bivariate matroid polynomial
that counts the lattice points in the Minkowski sum of the matroid polytope and scalings of the above simplices
contains the same information as the Tutte polynomial.
To be precise, either polynomial is an evaluation of the other.

K\'alm\'an's interest in this view of activity arose from the question of 
enumerating spanning trees of bipartite graphs according to their vector of degrees at the vertices of one colour.  
The set of such vectors is rarely a matroid polytope, but it is always a \emph{polymatroid} polytope.
Polymatroids were introduced in Edmonds' work in optimisation circa 1970, as an extension of matroids formed by
relaxing one matroid rank axiom so as to allow singletons to have arbitrarily large rank.
Postnikov's \emph{generalised permutohedra} \cite{postnikov} are virtually the same object.
(Integer generalised permutohedra are exactly translates of polymatroid polytopes,
and they are polymatroid polytopes just when they lie in the closed positive orthant.)

One would like an analogue of the Tutte polynomial for polymatroids.
Unlike the case for matroids, the simultaneous generating function for internal and external activity does not answer to this desire,
as it is not even an invariant of a polymatroid: 
Example \ref{bivaractivity} shows that it varies with different choices of ordering of the ground set.
Only the generating functions for either activity statistic singly are in fact invariants, 
as K\'alm\'an proved, naming them $I_r(\xi)$ and $X_r(\eta)$.
Our lattice point counting polynomial can be applied straightforwardly to a polymatroid,
and we prove in Theorem~\ref{thm:activity} that it specialises to $I_r(\xi)$ and~$X_r(\eta)$.  
That is, the invariant we construct is a bivariate analogue of K\'{a}lm\'{a}n's activity polynomials, which answers a question of~\cite{kalman}.

In seeking to generalise the Tutte polynomial to polymatroids one might be interested in other properties than activity, especially its universal property with respect to deletion-contraction invariants.
Like matroids, polymatroids have a well-behaved theory of \emph{minors}, analogous to graph minors:
for each ground set element $e$ one can define a deletion and contraction,
and knowing these two determines the polymatroid.  
The deletion-contraction recurrence for the Tutte polynomial reflects this structure.
The recurrence has three cases,
depending on whether $e$ is a loop, a coloop, or neither of these.
The number of these cases grows quadratically with the maximum possible rank of a singleton,
making it correspondingly hairier to write down a universal invariant.
In 1993 Oxley and Whittle \cite{whittle} addressed the case of \emph{$\it{2}$-polymatroids},
where singletons have rank at most~2, and prove that 
the corank-nullity polynomial is still universal for a form of deletion-contraction recurrence. 
The general case, with a natural assumption on the coefficients of the recurrence but one slightly stronger than  in~\cite{whittle},
is addressed using coalgebraic tools in~\cite{DFM},
fitting into the tradition of applications of coalgebras and richer structures in combinatorics:
see \cite{AA,KMT} for other work of particular relevance here.

For our polymatroid invariant we give a recurrence that involves not just the deletion and contraction,
but a whole array of ``slices'' of which the deletion and the contraction are the extremal members (Theorem~\ref{delcont}).
We do not know a recurrence relation where only the deletion and the contraction appear.

We would be remiss not to mention the work of K\'alm\'an and Postnikov \cite{kalman2} proving the central conjecture of \cite{kalman},
that swapping the two colours in a bipartite graph leaves $I_r(\xi)$ unchanged.
Their proof also exploits Ehrhart-theoretic techniques, but the key polytope is the \emph{root polytope} of the bipartite graph.
We expect that it should be possible to relate this to our machinery via the Cayley trick.
Oh \cite{oh} has also investigated a similar polyhedral
construction, as a way of proving Stanley's pure O-sequence
conjecture for cotransversal matroids.  

This paper is organised as follows. 
Section~\ref{sec:prelim} introduces the definitions of our main objects.
In Section \ref{sec:hints} we begin by explaining the construction of our polynomial for matroids, followed by how this is related to the Tutte polynomial (our main theorem, Theorem \ref{lpg}). 
In Section \ref{sec:coeffs}, we give a geometric interpretation of the coefficients of our polynomial, by way of a particular subdivision of the relevant polytope, which has a simple interpretation in terms of Dawson partitions. 
In Section~\ref{sec:polymatroids}, we discuss the extension to polymatroids,
including properties our invariant satisfies in this generality.
Section~\ref{sec:kalman} is dedicated to the relationship to K\'alm\'an's univariate activity invariants.

\subsection*{Acknowledgments}

We thank Tam\'as K\'alm\'an, Madhusudan Manjunath, and Ben P.\ Smith for fruitful discussions,
Iain Moffatt for very useful expository advice on a draft,
and the anonymous reviewers for FPSAC 2016 for their suggested improvements. 
The authors were supported by EPSRC grant EP/M01245X/1.

\section{Preliminaries}
\label{sec:prelim}

We assume the reader has familiarity with basic matroid terminology, and
recommend \cite{Oxley} as a reference for this material.
Given a set $E$, let $\mathcal P(E)$ be its power set.

\begin{dfn}
A \emph{polymatroid} $M=(E,r)$ on a finite \emph{ground set} $E$ 
consists of the data of a rank function $r:\mathcal{P}(E)\rightarrow \mathbb{Z^+}\cup\{0\}$ such that, for $X,Y\in\mathcal{P}(E)$, the following conditions hold:
\begin{itemize}
\item[{\rm P1}.] $r(\emptyset)=0$
\item[{\rm P2}.] If $Y\subseteq X$, then $r(Y)\leq r(X)$
\item[{\rm P3}.] $r(X\cup Y)+r(X\cap Y)\leq r(X)+r(Y)$
\end{itemize}
\end{dfn}
A \emph{matroid} is a polymatroid such that $r(i)\leq1$ for all $i\in E$.
Like matroids, polymatroids can be defined \emph{cryptomorphically} in other equivalent ways,
the way of most interest to us being as polytopes (Definition~\ref{dfn:polytope}).
We will not be pedantic about which axiom system we mean when we say ``polymatroid''.

The following three 
definitions of activity for polymatroids are from \cite{kalman}.

\begin{dfn}
A vector ${\bf x}\in\mathbb Z^E$ is called a \emph{base} of a polymatroid $M=(E,r)$ if ${\bf x}\cdot{\bf e}_E=r (E)$ 
and ${\bf x}\cdot{\bf e}_S\leq r (S)$ for all subsets $S\subseteq E$.
\end{dfn}

Let $\mathcal B_M$ be the set of all bases of a polymatroid $M=(E,r)$.

\begin{dfn}
A \emph{transfer} is possible from $u_1\in E$ to $u_2\in E$ 
in the base ${\bf x}\in \mathcal B_M\cap\mathbb{Z}^E$
if by decreasing the $u_1$-component of ${\bf x}$ by $1$ and increasing its $u_2$-component by $1$ we get another base.
\end{dfn}

Like matroids, polymatroids have a base exchange property \cite[Theorem 4.1]{hibi}.
If $\bf x$ and $\bf y$ are in $\mathcal B_M$ and ${\bf x}_i>{\bf y}_i$ for some $i\in E$,
then there exists $l$ such that ${\bf x}_l< {\bf y}_l$ and ${\bf x}-{\bf e}_i+{\bf e}_l$ is again in $\mathcal B_M$, 
or in other words, such that a transfer is possible from $i$ to $l$ in~$\bf x$.

Fix a total ordering of the elements of~$E$.

\begin{dfn}\label{dfn:activity}
\begin{enumerate}[i.]
\item We say that $u\in E$ is \emph{internally active} with respect to the base $x$ if no transfer is possible in~$x$ \emph{from} $u$ to a smaller element of $E$.
\item We say that $u\in E$ is \emph{externally active} with respect to $x$ if no transfer is possible in~$x$ \emph{to} $u$ from a smaller element of $E$.
\end{enumerate}
\end{dfn}

For $x\in \mathcal B_M\cap\mathbb{Z}^E$\!, let the set of internally active elements with respect to $x$ be denoted with $\mathrm{Int}(x)$, and let $\iota(x)=|\mathrm{Int}(x)|$;
likewise, let the set of externally active elements be denoted with $\mathrm{Ext}(x)$
and $\varepsilon(x)=|\mathrm{Ext}(x)|$.
Let $\ol{\iota}(x),\ol{\varepsilon}(x)$ denote 
the respective numbers of inactive elements. 

When $M$ is a matroid, the following definitions of activity are more commonly used,
analogous to Tutte's original formulation using spanning trees of graphs. 
\begin{dfn}\label{dfn:activity 2}
Take a matroid $M=(E,r)$. Let $B$ be a basis of $M$. 
\begin{enumerate}[i.]
\item We say that $e\in E-B$ is \emph{externally active} with respect to $B$ if $e$ is the smallest element in the unique circuit contained in $B\cup e$, with respect to the ordering on $E$.
\item We say that $e\in B$ is \emph{internally active} with respect to $B$ if $e$ is the smallest element in the unique cocircuit in $(E\setminus B)\cup e$.
\end{enumerate}
\end{dfn}

\begin{remark}\label{rem:activity}
In the cases where it is set forth, namely $e\in E-B$ for external activity and $e\in B$ for internal activity,
Definition~\ref{dfn:activity 2} agrees with Definition~\ref{dfn:activity}.
But it will be crucial that we follow Definition~\ref{dfn:activity} where Definition~\ref{dfn:activity 2} doesn't apply:
when $M$ is a matroid and $B$ a basis thereof, 
we consider all elements $e\in B$ externally active, and all elements $e\in E-B$ internally active, with respect to~$B$.
\end{remark}

\begin{dfn}\label{dfn:Tutte}
Let $M=(E,r)$ be a matroid with ground set $E$ and rank function $r:\mathcal{P}(E)\rightarrow \mathbb{Z^+}\cup\{0\}$. The \emph{Tutte polynomial} of $M$ is 
\begin{equation}\label{eq:Tutte}
    T_M(x,y) = \sum_{S\subseteq E} (x-1)^{r(M)-r(S)}(y-1)^{|S|-r(S)}.
\end{equation}    
    \end{dfn}

The presentation of the Tutte polynomial in Definition~\ref{dfn:Tutte}
is given in terms of the \emph{corank-nullity polynomial}:
up to a change of variables, it is the generating function for subsets $S$ of the
ground set by their corank $r(M)-r(S)$ and nullity $|S|-r(S)$. 
When $M$ is a matroid, the Tutte polynomial is equal to a generating function for activities:
\begin{equation}\label{eq:act}
T_M(x,y)=\frac1{x^{|E|-r(E)}y^{r(E)}}\sum_{B\in\mathcal{B}_M} x^{\iota(B)}y^{\varepsilon(B)}.
\end{equation}
The unfamiliar denominator in this formula appears on account of Remark~\ref{rem:activity}.
Although on its face the right hand side of the formula depends on the ordering imposed on~$E$,
its equality with the right hand side of equation~\ref{eq:Tutte} shows that there is no such dependence.
For polymatroids, activity invariants can be defined as well: see Definition~\ref{activity} and following discussion.

In this paper we will be principally viewing polymatroids as polytopes. 
These polytopes live in the vector space $\mathbb R^E$,
where the finite set $E$ is the ground set of our polymatroids.
For a set $U\subseteq E$, let ${\bf e}_U\in\mathbb R^E$ be the indicator vector of $U$,
and abbreviate ${\bf e}_{\{i\}}$ by ${\bf e}_{i}$.
Let $r:\mathcal{P}(E)\rightarrow\mathbb{Z}^+\cup\{0\}$ be a rank function,
and $M=(E,r)$ the associated polymatroid. 
The \emph{extended polymatroid} of $M$ is defined to be the polytope
\begin{displaymath}EP(M)=\{{\bf x}\in\mathbb{R}^E \ | \  {\bf x}\geq 0\mbox{ and }
{\bf x}\cdot{\bf e}_U\leq r(U) \ \mathrm{for \ all} \ U\subseteq E\}.\end{displaymath}

\begin{dfn}\label{dfn:polytope}
The \emph{polymatroid (base) polytope} of $M$
is a face of the extended polymatroid:
\[P(M) = EP(M)\cap\{{\bf x}\in\mathbb{R}^E \ | \  {\bf x}\cdot{\bf e}_E= r(E)\}
= \operatorname{conv} \mathcal B_M.\]
\end{dfn}
Either one of these polytopes contains all the information in the rank function.
In fact, they can be used as cryptomorphic axiomatisations of polymatroids:
a polytope whose vertices have nonnegative integer coordinates is a polymatroid polytope
if and only if all its edges are parallel to a vector of the form ${\bf e}_{i}-{\bf e}_{j}$ for some $i,j\in E$.
Extended polymatroids permit a similar characterisation;
moreover, they can be characterised as those polytopes over which a greedy algorithm
correctly optimises every linear functional with nonnegative coefficients, 
which was the perspective of their inventor Edmonds \cite{edmonds}.

\section{Our invariant}
\label{sec:hints}

This section describes the construction of our matroid polynomial, to be denoted $Q'_M$, 
which counts the lattice points of a particular Minkowski sum of polyhedra,
and explains its relation to the Tutte polynomial.
In anticipation of Section~\ref{sec:polymatroids} we set out the definition in the generality of polymatroids.

\subsection{Construction}
\label{sec:const}

Let $\Delta$ be the standard simplex in $\mathbb R^E$ of dimension $|E|-1$,
that is 
\[\Delta = \operatorname{conv}\{\mathbf e_{i}:i\in E\},\]
and $\nabla$ be its reflection through the origin,
$\nabla = \{-x : x\in\Delta\}$.
The faces of $\Delta$ are the polyhedra
\[\Delta_S = \operatorname{conv}\{\mathbf e_{i}:i\in S\}\]
for all nonempty subsets $S$ of~$E$; similarly, 
the faces of $\nabla$ are the polyhedra $\nabla_S$ given as the reflections of the $\Delta_S$.
(We exclude the empty set as a face of a polyhedron.)

We consider $P(M)+u\Delta+t\nabla $ where $M=(E,r)$ is any polymatroid and $u,t\in\mathbb{Z}^+\cup\{0\}$. 
We are interested in the lattice points in this sum. 
These can be interpreted as the vectors
that can be turned into bases of~$M$ by incrementing a coordinate $t$ times
and decrementing one $u$ times.
By Theorem 7 of \cite{mcmullan}, the number 
\begin{equation}\label{qdef}
Q_{M}(t,u):=\#(P(M)+u\Delta+t\nabla )\cap\mathbb{Z}^E
\end{equation}
of lattice points in the sum
is a polynomial in $t$ and~$u$, of degree $\dim(P(M)+u\Delta+t\nabla) = |E|-1$.
We will mostly work with this polynomial after a change of variables:
letting the coefficients $c_{ij}$ be defined by \begin{displaymath}Q_{M}(t,u)=\sum_{i,j} c_{ij}\binom{u}{j}\binom{t}{i},\end{displaymath}
we use these to define the polynomial \begin{displaymath}Q'_M(x,y)=\sum_{ij}c_{ij}(x-1)^i(y-1)^j.\end{displaymath}
The change of variables is chosen so that
applying it to $\#(u\Delta_X+t\nabla_Y)$ yields $x^iy^j$,
where $\Delta_X$ and $\nabla_Y$ are faces of $\Delta$ and $\nabla$ 
of respective dimensions $i$ and~$j$.  This will allow for a
combinatorial interpretation of the coefficients of~$Q'$ in Theorem~\ref{coeffs}.

\subsection{Relation to the Tutte polynomial}
\label{tutte}

For the remainder of Section~\ref{sec:hints}, we assume that $M$ is a matroid.
The main theorem of this section is that $Q'_M(x,y)$ is an evaluation of the Tutte polynomial, and in fact one that contains
precisely the same information.  
As such, the Tutte polynomial can be evaluated by lattice point 
counting methods.

\begin{restatable}{thm}{lpg}
\label{lpg}
Let $M=(E,r)$ be a matroid. Then
\begin{multline*}
T_M(x,y)=(xy-x-y)^{|E|}(-x)^{r(M)}(-y)^{|E|-r(M)}\cdot\\ \sum_{u,t\geq 0} Q_M(t,u)\cdot\left(\dfrac{y-xy}{xy-x-y}\right)^t\left(\dfrac{x-xy}{xy-x-y}\right)^u
\end{multline*}
\end{restatable}

Our proof of Theorem~\ref{lpg} arrives first at the relationship between $Q'_M(x,y)$ and the Tutte polynomial.

\begin{thm}\label{thm:T to Q}
Let $M=(E,r)$ be a matroid. Then we have that
\begin{equation}
Q'_M(x,y)=\dfrac{x^{|E|-r(M)}y^{r(M)}}{x+y-1}\cdot T_M\left(\dfrac{x+y-1}{y},\dfrac{x+y-1}{x}\right)
\end{equation}
\end{thm}

An observation is in order before we embark on the proof.
Since lattice points and their enumeration are our foremost concerns in this work,
we prefer not to have to think of the points of our polyhedra with non-integral coordinates.
It is the following lemma that lets us get away with this.
\begin{lemma}[{\cite[Corollary 46.2c]{Schrijver}}]\label{712}
Let $P$ and $Q$ be generalised permutohedra whose vertices are lattice points. 
Then if $x\in P+Q$ is a lattice point, there exist lattice points $p\in P$ and $q\in Q$ such that $x=p+q$.
\end{lemma}
By repeated use of the lemma, if $q\in u\Delta_{X}+{\bf e}_B+t\nabla_{Y}$ is a lattice point,
then $q$ has an expression of the form
\[
q = {\bf e}_B+{\bf e}_{i_1}+\cdots+{\bf e}_{i_t}-{\bf e}_{j_1}-\cdots-{\bf e}_{j_u}.
\]
Since all of the summands in $u\Delta_{X}+{\bf e}_B+t\nabla_{Y}$ are translates of matroid polytopes,
where the scalings are treated as repeated Minkowski sums,
this can also be proved using the matroid partition theorem, as laid out by Edmonds \cite{edmondspartition}.

\begin{proof}[Proof of Theorem~\ref{thm:T to Q}]
Let $q=e_B+{\bf e}_{x_1}+\cdots+{\bf e}_{x_i}-{\bf e}_{y_1}-\cdots -{\bf e}_{y_j}$ be a point in $P(M)+i\Delta+j\nabla$, where ${\bf e}_B\in P(M)$,
and Lemma~\ref{712} guarantees the existence of some such expression for any~$q$. 
We say that the expression for $q$ has a \emph{cancellation} if ${x_k}={y_l}$ for some $k,l$. Let $k$ be the number of cancellations in the expression for $q$, allowing a summand to appear in only one cancellation. For instance, if ${x_k}={y_l}={y_m}$ is the complete set of equalities, there is one cancellation, while ${x_k}={x_n}={y_l}={y_m}$ would give two. We will partition the set of lattice points of $P(M)+i\del+j\nab$ according to how many coordinates are non-negative, and then construct $Q$ by counting the lattice points in each part of the partition. Given a lattice point $q$, let $S=\{1\leq i\leq n \ | \ q_i> 0\}$. In order to construct ${\bf e}_B$ from ${\bf e}_S$ we need to use $|S|-r(S)$ of the $\del$ summands: as $B$ is spanning, we must have used a $\del$-summand every time $|S|$ rises above $r(S)$. Similarly, we must use $r(M)-r(S)$ $\nab$-summands to account for any fall in rank. Now, we will set the remaining $\del$-summands equal to those coordinates already positive, that is, set them equal to indicator vectors of elements of $S$. The remaining $\nab$-summands we will set equal to indicator vectors of elements not in $S$. 

We will ensure, through choice of $B$ and $k$, that this is the largest $q$ (in terms of sum of coordinates) we can find given $i$ and $j$. There are two ways the expression can fail to be maximal in this sense:
\begin{itemize}
\itemsep0em
\item when we decrease $k$, we could construct $q$ using fewer $\del$ and $\nab$ summands, and
\item if we write $q$ using $B'$ where we have summands ${\bf e}_a-{\bf e}_b$ such that $(B'\cup a)-b$ is a valid basis exchange, we would again be able to construct $q$ using fewer summands. 
\end{itemize}

We will choose $B$ in the expression for $q$ and the maximal $k$ so that describing all lattice points $q$ can be done uniquely in the way described.

Now we have that $|S|$ is the number of non-negative coordinates in at least one point of $P(M)+i\Delta+j\nabla$, and all our positive summands of such a point are assigned to such coordinates. The sum of these summands must be $r(M)+i-k-|S|=i-k-\operatorname{null}(S)$. If we ensure that $|E-S|$ is the number of negative integers in the respective points of $P(M)+i\Delta+j\nabla$, summing over these sets $S$ will give a count of all lattice points.  By the reasoning above, we must have that the $|E-S|$ non-negative integers sum to $j-k-r(M)+r(S)=j-k-\operatorname{cork}(S)$. Thus,
$$ \#(P(M)+i\Delta+j\nabla)=\sum_S \sum_k[ \#(|S| \ \text{non-negative integers summing to} \ i-k-\operatorname{null}(S))$$
$$\qquad\qquad\qquad\qquad\times\#(|E-S| \ \text{non-negative integers summing to} \ j-k-\operatorname{cork}(S))]$$
which is
\begin{multline}
\label{points} \#(P(M)+i\Delta+j\nabla)=\sum_S \sum_k \binom{i-k+|S|-\operatorname{null}(S)-1}{|S|-1}\\\qquad\qquad\times\binom{j-k+|E-S|-\operatorname{cork}(S)-1}{|E-S|-1}.
\end{multline}


Now form the generating function $$\sum_{i,j} \#(P(M)+i\Delta+j\nabla)v^iw^j=\sum_i \sum_j \#(P(M)+i\Delta+j\nabla)v^iw^j.$$ 
Substituting Equation \ref{points} into the generating function gives
\begin{align*}
\sum_i \sum_j\sum_S \sum_{k\geq 0} &\binom{i-k+|S|-\operatorname{null}(S)-1}{|S|-1}v^{i-k}\\
&\times\binom{j-k+|E-S|-\operatorname{cork}(S)-1}{|E-S|-1}w^{j-k}(vw)^k.
\end{align*}
Using the identity $\sum\limits_i\binom{i+a}{b}x^i=\dfrac{x^{b+a}}{(1-x)^{b+1}}$ simplifies this to
$$\sum_S \sum_k \dfrac{v^{\nul(S)}}{(1-v)^{|S|}}\cdot\dfrac{w^{\cork(S)}}{(1-w)^{|E-S|}}\cdot (vw)^k$$
which we can write as
$$\sum_S \dfrac{v^{\nul(S)}}{(1-v)^{\nul(S)-\cork(S)+r(M)}}\cdot\dfrac{w^{\cork(S)}}{(1-w)^{\cork(S)-\nul(S)+|E|-r(M)}}\cdot\sum_k(vw)^k.$$
Collecting like exponents, we end up with
\begin{align}
\sum_{i,j} \#(P(M)+i\Delta+j\nabla)v^iw^j &=\dfrac{1}{1-vw}\cdot\dfrac{1}{(1-v)^{r(M)}(1-w)^{|E|-r(M)}}\nonumber\\&\qquad\qquad\times\sum_S \left(\dfrac{v(1-w)}{1-v}\right)^{\nul(S)}\left(\dfrac{w(1-v)}{1-w}\right)^{\cork(S)} \nonumber\\
&= \dfrac{1}{1-vw}\cdot\dfrac{1}{(1-v)^{r(M)}(1-w)^{|E|-r(M)}}\nonumber\\&\qquad\qquad\times T\left(\dfrac{w(1-v)}{1-w}+1,\dfrac{v(1-w)}{1-v}+1\right) \nonumber\\
& = \dfrac{1}{1-vw}\cdot\dfrac{1}{(1-v)^{r(M)}(1-w)^{|E|-r(M)}}\nonumber \\ &\qquad\quad\qquad\times T_M\left(\dfrac{1-vw}{1-w},\dfrac{1-vw)}{1-v}\right)
\end{align}
where $T_M$ is the Tutte polynomial of $M$. Now it remains to be shown that the left-hand side contains an evaluation of our polynomial $Q'_M$. 
Using our original definition of $Q_M$, Equation \eqref{qdef}, we have that
\begin{align*}
\sum_{i,j} \#(P(M)+i\Delta+j\nabla)v^iw^j &= \sum_{i,j,k,l}c_{kl}\binom{i}{l}\binom{j}{k}v^iw^j \\
&= \sum_{k,l} c_{kl}\cdot\dfrac{v^l}{(1-v)^{l+1}}\cdot\dfrac{w^k}{(1-w)^{k+1}}.
\end{align*}
If we let $\dfrac{w}{1-w}=x-1$ and $\dfrac{v}{1-v}=y-1$, then
\begin{align*}
\sum_{i,j} \#(P(M)+i\Delta+j\nabla)v^iw^j &= \sum_{k,l} c_{kl}\cdot\dfrac{v^l}{(1-v)^{l+1}}\cdot\dfrac{w^k}{(1-w)^{k+1}} \\
&= (1-v)(1-w)\sum_{k,l}c_{kl}(x-1)^k(y-1)^l\\
&= (1-v)(1-w)Q_M'(x,y).
\end{align*}
So, from Equation (5.8.2), we have that
\begin{align*}
(1-v)(1-w)Q'_M(x,y)=\dfrac{1}{1-vw}\cdot\dfrac{1}{(1-v)^{r(M)}(1-w)^{|E|-r(M)}}\cdot T_M\left(\dfrac{1-vw}{1-w},\dfrac{1-vw}{1-v}\right).
\end{align*}
Solving for $w$ and $v$ in terms of $x$ and $y$ gives that $w=\dfrac{x-1}{x},v=\dfrac{y-1}{y}$. Substitute these into the above equation to get 
\begin{equation}
Q_M'(x,y)=\dfrac{x^{|E|-r(M)}y^{r(M)}}{x+y-1}\cdot T_M\left(\dfrac{x+y-1}{y},\dfrac{x+y-1}{x}\right).
\end{equation}
\end{proof}

We can invert this formula by setting $x'=\dfrac{x+y-1}{y},y'=\dfrac{x+y-1}{x}$, rearranging, and then relabelling. 
\begin{thm}\label{thm:Q to T}
Let $M=(E,r)$ be a matroid. Then
\begin{equation}
T_M(x,y)=-\,\dfrac{(xy-x-y)^{|E|-1}}{(-y)^{r(M)-1}(-x)^{|E|-r(M)-1}}\cdot Q_M'(\dfrac{-x}{xy-x-y},\dfrac{-y}{xy-x-y})
\end{equation}
\end{thm}

We conjecture that there is a relationship between
our formula for the Tutte polynomial and
the algebro-geometric formula for the Tutte polynomial in \cite{speyer}.
The computations on the Grassmannian in that work are done 
in terms of~$P(M)$, the moment polytope of a certain torus orbit closure,
and $\Delta$ and $\nabla$ are the moment polytopes of the 
two dual copies of~$\mathbb P^{n-1}$, the $K$-theory ring of whose product
$\mathbb Z[x,y]/(x^n,y^n)$ is identified with the ambient ring of the Tutte polynomial.

\begin{ex}\label{ex:1}
Let $M$ be the matroid on ground set $[3]=\{1,2,3\}$ with
$\mathcal B_M = \{\{1\},\{2\}\}$.
When $u=2$ and $t=1$, the sum $P(M)+u\Delta+t\nabla $ is the polytope of Figure~\ref{fig:1}, with 16 lattice points.
\begin{figure}
\begin{center}
\includegraphics[scale=1]{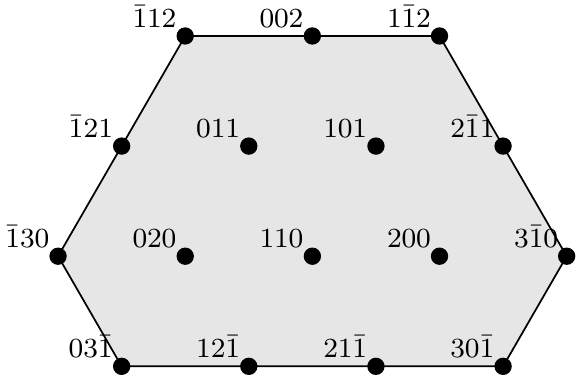}
\end{center}
\caption{The polytope $P(M)+u\Delta+t\nabla $ of Example~\ref{ex:1}.
The coordinates are written without parentheses or commas,
and $\bar1$ means $-1$.}\label{fig:1}
\end{figure}

To compute $Q_M(x,y)$, it is enough to count the lattice points in
$P(M)+u\Delta+t\nabla $ for a range of $u$ and $t$, and interpolate. 
Since $Q_M$ is a polynomial of degree~2, it is sufficient to take
$t$ and $u$ nonnegative integers with sum at most~2.
These are the black entries in the table below:
\begin{center}
\begin{tabular}{l|ccc}
$t$ $\setminus$ $u$ & 0 & 1 & 2 \\\hline
0 & 2 & 5 & 9 \\
1 & 5 & 10 & \textcolor{Gray}{16} \\
2 & 9 & \textcolor{Gray}{16} & \textcolor{Gray}{24}
\end{tabular}
\end{center}
The unique degree~$\leq2$ polynomial with these evaluations is
\[Q_M(t,u) = \binom t2 + 2tu + \binom u2 + 3t + 3u + 2,\]
so
\begin{align*}
Q'_M(x,y) &= (x-1)^2 + 2(x-1)(y-1) + (y-1)^2 + 3(x-1) + 3(y-1) + 2
\\&= x^2 + 2xy + y^2 - x - y.
\end{align*}
Finally, by Theorem~\ref{thm:Q to T},
\begin{align*}
T_M(x,y) &= -\,\frac{(xy-x-y)^2}{(-y)^0(-x)^1}\cdot
\left(\frac{y^2 + 2xy + x^2}{(xy-x-y)^2}+\frac{y + x}{xy-x-y}\right)
\\&= xy + y^2
\end{align*}
which is indeed the Tutte polynomial of~$M$.
\end{ex}

Given Theorem \ref{thm:Q to T}, we can now prove Theorem \ref{lpg}.

\lpg*

\begin{proof}
Consider the power series $\Sigma:=\sum\limits_{u,t\geq 0}Q_M(t,u)\,a^tb^u$. Note that 
$$\sum_{u,t\geq 0}\binom{t}{i}\binom{u}{j}\,a^tb^u=\frac{1}{ab}\cdot\left(\dfrac{a}{1-a}\right)^{i+1}\left(\dfrac{b}{1-b}\right)^{j+1}.$$ 
We can thus write $\Sigma$ as 
$$\frac{1}{ab}\sum_{i,j}c_{ij}\left(\dfrac{a}{1-a}\right)^{i+1}\left(\dfrac{b}{1-b}\right)^{j+1}.$$ 
Substituting $a=(v-1)/v$ and $b=(w-1)/w$ turns this into 
\begin{align*}
\Sigma &= \dfrac{vw}{(v-1)(w-1)}\sum_{i,j}c_{ij}(v-1)^{i+1}(w-1)^{j+1}
\\ &= vw \sum_{i,j}c_{ij}(v-1)^i(w-1)^j
\\ &= vw\, Q_M'(v,w).
\end{align*}
We can now apply Theorem \ref{thm:T to Q}:
\begin{align*}
&\mathrel{\phantom=}\sum_{u,t\geq 0} Q_M(t,u)\left(\dfrac{v-1}{v}\right)^t\left(\dfrac{w-1}{w}\right)^u \\
&= vw\, Q'_M(v,w)\\
&= \dfrac{v^{|E|-r(M)+1}w^{r(M)+1}}{v+w-1}\cdot T_M\left(\dfrac{v+w-1}{w},\dfrac{v+w-1}{v}\right).\\
\end{align*}
Substitute $v=-x/(xy-x-y)$ and $w=-y/(xy-x-y)$ to get the stated result.
\end{proof}

A further substitution and simple rearrangement gives the following corollary, included for the sake of completeness.

\begin{cor}
Let $M=(E,r)$ be a matroid. Then
\[\sum_{u,t\geq 0}Q_M(t,u)v^tw^u=\dfrac{1}{(1-v)^{|E|-r(M)}(1-w)^{r(M)}(1-vw)}\cdot T_M\left(\dfrac{1-vw}{1-v},\dfrac{1-vw}{1-w}\right).\]
\end{cor}

Being a Tutte evaluation, $Q'$ must have a deletion-contraction recurrence.
We record the form it takes.

\begin{prop}
\label{recurrence}
Let $M=(E,r)$ be a matroid with $|E|=n$. Then, for $e\in E$,
\begin{enumerate}[i.]
\item $Q'_M(x,y)=xQ_{M\backslash e}(x,y)+yQ'_{M/e}(x,y)$ when $e$ is not a loop or coloop, and
\item $Q'_M(x,y)=(x+y-1)Q'_{M/e}(x,y)=(x+y-1)Q'_{M\backslash e}(x,y)$ otherwise.
\end{enumerate}
\end{prop}

\begin{proof}Part \emph{ii} is a consequence of Proposition~\ref{prop:direct sum} below (which does not depend on the present section).
When $e$ is a (co)loop, $M = M_e\oplus M\backslash e = M_e\oplus M/e$,
where $M_e$ is the restriction of $M$ to~$\{e\}$ (or the equivalent contraction).

For part \emph{i}, recall that if $e$ is neither a loop nor a coloop, then $E(M\backslash e)=E-e=E(M/e)$, $r(M\backslash e)=r(M)$, and $r(M/e)=r(M)-1$. Take the equation $T_M(x,y)=T_{M\backslash e}(x,y)+T_{M/e}(x,y)$ and rewrite it in terms of $Q'$, as per Theorem~\ref{thm:Q to T}:
\begin{multline*}
-\dfrac{(xy-x-y)^{n-1}}{(-y)^{r(M)-1}(-x)^{n-r(M)-1}}\cdot Q'_M(x,y)= \\
-\dfrac{(xy-x-y)^{n-2}}{(-y)^{r(M)-1}(-x)^{n-r(M)-2}}\cdot Q'_{M\backslash e}(x,y) -\dfrac{(xy-x-y)^{n-2}}{(-y)^{r(M)-2}(-x)^{n-r(M)-1}}\cdot Q'_{M/e}(x,y)
\end{multline*}
Multiplying through by $-\dfrac{(-y)^{r(M)-1}(-x)^{n-r(M)-1}}{(xy-x-y)^{n-1}}$ gives the result.
\end{proof}

\section{Coefficients}
\label{sec:coeffs}
Some coefficients of the Tutte polynomial provide structural information about the matroid in question. Let $b_{i,j}$ be the coefficient of $x^iy^j$ in $T_M(x,y)$. The best-known case is that $M$ is connected only if $b_{1,0}$, known as the \emph{beta invariant}, is non-zero; moreover, $b_{1,0} = b_{0,1}$ when $|E|\geq 2$. 
Not every coefficient yields such an appealing result, though of course
they do count the bases with internal and external activity of fixed sizes. 
In like manner, 
we are able to provide a enumerative interpretation of the coefficients of $Q'_M(x,y)$, which is the focus of this section.

In order to do this, we will make use of a regular mixed subdivision $\mathcal{F}$ of $u\Delta +P(M)+t\nabla$. 
Let $\alpha_1<\cdots <\alpha_n$ and $\beta_1<\cdots <\beta_n$ be positive reals. 
Our regular subdivision will be that determined by 
projecting the ``lifted'' polytope 
\[\mathit{Lift}=\conv\{(u{\bf e}_i,\alpha_i)\}+(P(M)\times \{0\})+\conv\{(-t{\bf e}_i,\beta_i)\}\subseteq\mathbb R^E\times\mathbb R\]
to $\mathbb R^E$.
Let $\mathfrak{F}$ be the set of ``lower'' facets of~$\mathit{Lift}$ which maximise some linear function $\langle a,x\rangle$, where $a\in(\mathbb R^E\times\mathbb R)^*$ is a linear functional with last coordinate $a_{n+1}=-1$. 
For each face $F\in\mathfrak{F}$, let $\pi(F)$ be its projection back to $\mathbb{R}^n$. 
Now $\mathcal{F}:=\{\pi(F) \ | \ F\in\mathfrak{F}\}$ is a regular subdivision of $u\Delta+P(M)+t\nabla$.
We will write $\mathcal F$ as $\mathcal F(t,u)$ when we need to make the dependence on the parameters explicit.
Note however that the structure of the face poset of~$\mathcal{F}$ 
does not depend on $t$ and $u$ as long as these are positive. 

Since $\mathcal{F}$ is a mixed subdivision of $u\Delta +P(M)+t\nabla$, 
each of its cells bears a canonical decomposition as a Minkowski sum of 
a face of $u\Delta$, a face of $P(M)$, and a face of $t\nabla$.
When we name a face of $\mathcal F$ as a sum of three polytopes $F+G+H$, 
we mean to invoke this canonical decomposition.
These decompositions are compatible between faces: 
if $m_a(P)$ denotes the face of a polytope~$P$ on which a linear functional $a$ is maximised,
then the canonical decomposition for $m_a(F+G+H)$ is $m_a(F)+m_a(G)+m_a(H)$.

We now state the main result of this section:

\begin{thm}
\label{coeffs}
Take the regular mixed subdivision $\mathcal F$ of $u\Delta+P(M)+t\nabla$ as described above. 
The unsigned coefficient $|[x^iy^j]Q'_M|$ counts the cells $F+G+H$ of~$\mathcal F$ where $i=dim(F)$, $j=dim(H)$,
and $G$ is a vertex of $P(M)$ and there exists no cell $F+G'+H$ where $G'\supsetneq G$. 
\end{thm}

The key fact in the proof is the following. 

\begin{dfn}
A maximal cell $F+G+H$ of the mixed subdivision $\mathcal F$ is a \emph{top degree face}
when $G$ is a vertex of $P(M)$.
\end{dfn}


\begin{prop}\label{prop:top degree}
In the subdivision $\mathcal{F}$, 
each of the lattice points of $u\Delta+P(M)+t\nabla$ lies in a top degree face.
\end{prop}

To expose the combinatorial content of this proposition, we need to describe
the top degree faces more carefully. All top degree faces are of dimension $|E|-1$
and have the form $u\Delta_{X}+{\bf e}_B+t\nabla_{Y}$. 
By Lemma~\ref{712}, if $p\in u\Delta_{X}+{\bf e}_B+t\nabla_{Y}$ is a lattice point,
then $p$ has an expression of the form
\begin{equation}\label{eq:tdf sum}
    p = {\bf e}_B+{\bf e}_{i_1}+\cdots+{\bf e}_{i_t}-{\bf e}_{j_1}-\cdots-{\bf e}_{j_u}.
\end{equation}
The subdivision $\mathcal F$ determines a height function $h(x)$
on the lattice points $x$ of~$u\Delta +P(M)+t\nabla$,
where $h(x)$ is the minimum real number such that $(x,h(x))\in\mathit{Lift}$.
This height function is 
\[h(x):=\text{min}\{\alpha_{i_1}+\cdots+\alpha_{i_t}+\beta_{j_1}+\cdots+\beta_{j_u} 
\ | \ x-{\bf e}_{i_1}-\cdots-{\bf e}_{i_t}+{\bf e}_{j_1}+\cdots+{\bf e}_{j_u}\in \mathcal B_M\}.\]
If $x$ is a lattice point of a top-degree face then choosing the $i_k$ and $j_l$ in accord with \eqref{eq:tdf sum} achieves the minimum.

Let $\Pi=u\Delta_{X}+{\bf e}_B+t\nabla_{Y}$ be a top-degree face.
If $i$ and $j$ were distinct elements of $X\cap Y$,
then $\Pi$ would have edges of the form $\conv\{x,x+k({\bf e}_i-{\bf e}_j)\}$ whose preimages in the corresponding lower face of $\mathit{Lift}$ were not edges, 
since they would contain the sum of the nonparallel segments 
$\conv\{(u{\bf e}_i,\alpha_i),(u{\bf e}_j,\alpha_j)\}$ and $\conv\{(-t{\bf e}_i,\beta_i),(-t{\bf e}_j,\beta_j)\}$.
Therefore we must have $|X\cap Y|\leq1$.
Together with the fact that the dimensions of $\Delta_{X}$ and~$\nabla_{Y}$ sum to $|E|-1$, 
this implies that $X\cup Y=E$ and $|X\cap Y|=1$.  In fact 
the conditions on the $\alpha$ and $\beta$ imply that $X\cap Y=\{1\}$,
because replacing equal subscripts $i_k=j_l>1$ by~$1$ in the definition of $h(x)$ decreases the right hand side.

We thus potentially have $2^{|E|-1}$ top degree faces, one for each remaining valid choice of $X$ and $Y$ -- each element except $1$ is either in $X$ but not $Y$, or it is in $Y$ but not $X$.
In fact, all $2^{|E|-1}$ of these do appear in~$\mathcal{F}$.

\begin{lemma}\label{lem:B}
Take subsets $X$ and $Y$ of~$E$ with $X\cup Y=E$ and $X\cap Y=\{1\}$.
There is a unique basis $B$ such that 
$u\Delta_{X}+{\bf e}_B+t\nabla_{Y}$ is a top-degree face. 
It is the unique basis $B$ such that
no elements of $X$ are externally inactive and no elements of $Y$ are internally inactive with respect to $B$, with reversed order on $E$.
\end{lemma}

The basis $B$ can be found using the simplex algorithm for linear programming
on $P(M)$, applied to a linear functional constructed from the $\alpha$ and $\beta$
encoding the activity conditions.
This procedure can be completely combinatorialised, giving a way to start from a randomly chosen initial basis and make a sequence of exchanges which yields a unique output $B$ regardless of the input choice.
The proof is as follows.

\begin{proof}
Choose any basis, $B_0$, and order the elements $b_1,\ldots, b_r$ lexicographically. Perform the following algorithm to find the basis $B$.
The algorithm makes a sequence of replacements of the elements of~$B$, of two kinds, until it is unable to make any more.

%
%
%

\medskip

\begin{algorithm}\label{alg:1}\mbox{}
\begin{enumerate}[(1)]
    \item[\textbf{Input:}] a basis $B_0$ of~$M$.
    \item[\textbf{Output:}] the basis $B$ of~$M$ called for in the lemma.
    \item Let $i=0$.
    \item Attempt to produce new bases $B_1,B_2,\ldots$ as follows. 
    Let the elements of $B_i$ be $b_1,\ldots,b_r$, where $b_1<\cdots<b_r$.  For each $j=1,\ldots,r$:
    \begin{enumerate}[(a)]
        \item If there exists $x\in X$ greater than $b_j$ such that $B_i\setminus\{b_j\}\cup\{x\}$ is a basis of~$M$,
        then choose the maximal such $x$, let $B_{i+1} = B_i\setminus\{b_j\}\cup\{x\}$, increment $i$, and repeat step~(2).
        Call this a check of type (a). 
        \item If not, and $b_j\in Y$, and there exists $z$ less than $b_j$ such that $B_i\setminus\{b_j\}\cup\{z\}$ is a basis of~$M$,
        then choose the minimal such $z$, let $B_{i+1} = B_i\setminus\{b_j\}\cup\{z\}$, increment $i$, and repeat step~(2).
        Call this a check of type (b).
    \end{enumerate}
    \item Terminate and return $B=B_i$.
\end{enumerate}
\end{algorithm}

The remainder of the proof analyses this algorithm.
Let $\gamma_1,\ldots,\gamma_n\in\mathbb{R}$ be such that $0=|\gamma_1|\ll \cdots \ll |\gamma_n|$, and $\gamma_a>0$ if $a\in X$ while $\gamma_a<0$ if $a\in Y$.

\begin{sublemma}
\label{claim3}
Let $B_i$ and $B_{i+1}$ be two bases of $M$ found consecutively by the algorithm. Then $\sum\limits_{a\in B_i}\gamma_a < \sum\limits_{a\in B_{i+1}}\gamma_a $ for all $i$. That is, the sum $\sum\limits_{a\in B}\gamma_a $ is increasing with the algorithm.
\end{sublemma}

\begin{subproof}[Proof of Claim~\ref{claim3}]
Moves of type (a) replace an element $b$ of a basis with a larger element $c$ in $X$, so regardless of whether $b$ was in $X$ or $Y$, this must increase the sum as $\gamma_b>0$. Moves of type (b) replace an element $y\in Y$ in the basis with a smaller element $d$. If $d\in Y$, we are replacing $\gamma_y$ with a smaller negative, as $|\gamma_d|\ll|\gamma_y|$. If $d\in X$, we are replacing a negative $\gamma_y$ with a positive $\gamma_d$. So $\sum\limits_{a\in B}\gamma_a$ is increasing in every case.
\end{subproof}

We will write the symmetric difference of two sets $A$ and $B$ as $A\tri  B$. The next result follows as a corollary of the previous claim.

\begin{sublemma}
\label{delta}
$B\tri  Y$ written with the largest elements first is lexicographically increasing with the algorithm.
\end{sublemma}

\begin{subproof}[Proof of Claim~\ref{delta}]
A move of type (a) replaces an element of $B$ with a larger element in $X$. This puts a larger element into $B\tri Y$ than was in $B$ originally, and so must cause a lexicographic increase. A move of type (b) removes a $Y$ element in $B$ (and so, an element not in $B\tri  Y$), and puts  a smaller element $z$ into $B$. Removing the $Y$ element from $B$ adds it to $B\tri  Y$, and adding an element to a set cannot decrease the lexicographic order. If the smaller element $z$ is in $Y$, then this move removes $z$ from $B\tri  Y$. As we have replaced it with a larger element, the lexicographic order of $B\tri  Y$ is increased. If $z$ is in $X$, the move adds $z$ to $B\tri  Y$, increasing the lexicographic order of $B\tri  Y$.
\end{subproof}

\begin{sublemma}
\label{claim4}
The algorithm described above terminates and gives an output independent of $B_0$.
\end{sublemma}

\begin{subproof}[Proof of Claim~\ref{claim4}]
Order all bases of the matroid based on the increasing lexicographic order of $B\tri Y$. As we chose the elements $\delta_a$ to be much greater than the previous element, only the largest element of $B\tri  Y$ determines the total ordering. We have shown in the previous corollary that this sequence is increasing with the algorithm. As there is a finite number of bases, there must be a greatest element, and thus the algorithm terminates.

We now need to show that there is a unique basis for which the algorithm can terminate.

The structure $(E,\{B_i\tri  Y \ | \ B_i\in\mathcal{B}_M\})$ is what is known as a \emph{delta-matroid}, a generalisation of a matroid allowing bases to have different sizes. This delta-matroid is a \emph{twist} of $M$ by the set $Y$ \cite{bouchet}. The subsets $B_i\tri  Y$ are called the feasible sets. It is a result of Bouchet (\cite{bouchet}) that feasible sets of largest size form the bases of a matroid. 


In order to show uniqueness of termination bases, we will first show that if $B$ is a termination basis and $B\tri  Y$ is not of largest size, then any basis $B'$ with $|B'\tri  Y|>|B\tri  Y|$ is terminal. Suppose this is not the case. As $|B\tri  Y|=|B|+|Y|-2|B\cap Y|$, this requires that $|B'\cap Y|<|B\cap Y|$.

As $B'$ is not a termination basis, there is either an element $b\in B'$ such that $(B'-b)\cup c\in\mathcal{B}$ for some element $c\in X$, where $c>b$, or there is an element $b\in Y\cap B$  such that $(B'-b)\cup a\in\mathcal{B}$ for some $a<b$. If we have $b,c\in X$, $|((B'-b)\cup c)\cap Y|=|B'\cap Y|$. If $b\in Y$, $|((B'-b)\cup c)\cap Y|<|B'\cap Y|$. If $a,b\in Y$, then $|((B'-b)\cup a)\cap Y|=|B'\cap Y|$. Finally, if $b\in Y$ and $a\in X$, then $|((B'-b)\cup a)\cap Y|<|B'\cap Y|$. In every case we have a contradiction.

As the algorithm terminates, we know that after a finite number of such exchanges, we produce $B$ from $B'$. Let the bases constructed in each step form a chain $$B',B_1,B_2,\ldots,B_n, B.$$ From above, we have that $|B'\cap Y|\geq |B_1\cap Y|\geq\cdots\geq |B_n\cap Y|\geq |B\cap Y|$. This contradicts the initial assumption that $|B'\cap Y|<|B\cap Y|$. 
%

Now assume the algorithm can terminate with two bases $B_1,B_2$. Take $B_1\tri  Y$ and $B_2\Delta Y$, and choose the earliest element $b\in B_1\tri  Y-B_2\tri Y$ (assuming this comes lexicographically first in $B_1\tri  Y$). If $b\in X$, then $b\in B_1-B_2$. If $b\in Y$, then $b\in B_2-B_1$. Similarly, if $c\in B_2\tri  Y-B_1\tri  Y$, if $c\in X$ then $c\in B_2-B_1$, or if $c\in Y$ then $c\in B_1-B_2$.

Apply the delta-matroid exchange algorithm to $B_1\tri  Y$ and $B_2\tri  Y$ to get that $(B_1\tri  Y)\tri \{b,c\}$ is a feasible set, for some element $c\in (B_1\tri  Y)\tri (B_2\tri  Y)$. Given we have a twist of a matroid, we must have that $(B_1\tri  Y)\tri \{b,c\}=B_3\tri  Y$ for some basis $B_3$, and so $|(B_1\tri \{b,c\}|=|B_3|=|B_1|$ as $\tri $ is associative. This means we must have that exactly one of $\{b,c\}$ is in $B_1$. If $b\in X$, $(B_1\tri  Y-b)\cup c=((B_1-b)\cup c)\tri  Y$, so $(B_1-b)\cup c\in\mathcal{B}$ and we must have $c\in X$ by the above paragraph. As $b$ was the earliest element different in either basis, we must have $c>b$, and so $B_1$ was not a termination basis of the original algorithm.  If $b\in Y$, $(B_1\tri  Y-b)\cup c=B_1\tri  ((Y-b)\cup c)$. But we cannot change $Y$, so must have $[(B_1-c)\cup b]\tri Y$ and $c\in Y$. This means that again $B_1$ was not a termination basis, as we are replacing an element of $Y$ with a smaller one. This completes the proof of Claim~\ref{claim4}.
\end{subproof}
This, in turn, completes the proof of Lemma \ref{lem:B}.
\end{proof}

Before we can get to the proof of Theorem \ref{coeffs}, we first need two results on how these top degree cells interact.
Note that in the service of readability we write $1$ instead of $\{1\}$ in subscripts. 
When we say that a polytope contains a basis $B$, we mean that it contains the indicator vector ${\bf e}_B$. 

\begin{lemma}
\label{meet}
Take two distinct partitions $(X_1,Y_1)$, $(X_2,Y_2)$ of $[n]\setminus\{1\}$. 
Let $B_1$, $B_2$ be the bases found by Algorithm~\ref{alg:1} such that we have top degree cells $T_i={\bf e}_{B_i}+\Delta_{1\cup X_i}+\nabla_{1\cup Y_i}$, $i\in\{1,2\}$. Suppose that $T_1\cap T_2\neq\emptyset$.  Then $B_1=B_2$.
\end{lemma}

\begin{proof}
As we have noted, the combinatorial type of the subdivision $\mathcal F(t,u)$ is independent of the values of $t$ and~$u$, as long as these are positive.
Also, if $t\leq t'$ and $u\leq u'$,
then each cell of $\mathcal F(t,u)$ is a subset of the corresponding cell of $\mathcal F(t',u')$,
up to translation of the latter by $(t'-t-u'+u){\bf e}_1$.
Thus if the top-degree cells indexed by $(X_1,Y_1)$ and $(X_2,Y_2)$ intersect, they will continue to intersect if $t$ or $u$ are increased.
So we may assume that none of $t$, $u$, $t-u$ lie in $\{-1,0,1\}$, by increasing $t$ and $u$ as necessary.

Because $\mathcal F$ is a cell complex, $T_1\cap T_2$ is a face of~$\mathcal F$, and it therefore contains a vertex $p$ of~$\mathcal F$.
For each $i=1,2$, the point $p$ is the sum of ${\bf e}_{B_i}$, a vertex of $u\Delta$, and a vertex of $t\nabla$.
Because every subset of the list $1,u,-t$ has a different sum,
$p$ can be decomposed as a zero-one vector plus a vertex of $u\Delta$ plus a vertex of $t\nabla$ in only one way,
and it follows that ${\bf e}_{B_1} = {\bf e}_{B_2}$.
\end{proof}

%

\begin{lemma}
\label{lattice}
Take two distinct partitions $P_1=(X_1,Y_1)$, $P_2=(X_2,Y_2)$ of $[n]\setminus\{1\}$ such that their corresponding top degree cells $T_1$ and $T_2$ contain a common point $p$. Now let $P_3=(X_3,Y_3)$ be a partition of $[n]\setminus\{1\}$ such that 
$X_1\cap X_2\subseteq X_3$ and $Y_1\cap Y_2\subseteq Y_3$. Then $p$ is in the top degree cell $T_3$ indexed by~$P_3$, 
and Algorithm~\ref{alg:1} finds the same basis $B^*$ for each of $P_1$, $P_2$ and $P_3$.
\end{lemma}

\begin{proof}
By Lemma \ref{meet}, we have $T_1={\bf e}_{B^*}+\Delta_{1\cup X_1}+\nabla_{1\cup Y_1}$ and $T_2={\bf e}_{B^*}+\Delta_{1\cup X_2}+\nabla_{1\cup Y_2}$ where the basis $B^*$ found by Algorithm~\ref{alg:1} is common to both expressions. 
The lexicographically greatest set of form $B\tri Y_1$ is that with $B=B^*$, likewise for $B\tri  Y_2$. 

Let $B'\tri Y_3$ be the lexicographically greatest set of form $B\tri Y_3$.  Our objective is to show that $B' = B^*$.
Assume otherwise for a contradiction, and let $e$ be the largest element in $B'\tri B^* = (B'\tri Y_3)\tri(B^*\tri Y_3)$.
By choice of~$B'$, we have $e\in B'\tri Y_3$ and $e\not\in B^*\tri Y_3$.
The latter implies that $e\not\in B^*\tri Y_i$ for at least one of $i=1,2$; without loss of generality, say $e\not\in B^*\tri Y_1$.
Then $B'\tri Y_1$ is lexicographically earlier than $B^*\tri Y_1$, because the former but not the latter contains $e$ and they agree in which elements greater than $e$ they contain.
This is the desired contradiction.

We conclude that $T_k = {\bf e}_{B^*}+\del_{1\cup X_k}+\nab_{1\cup Y_k}$ for each $k=1,2,3$.
Expanding $p-{\bf e}_{B^*}$ in the basis ${\bf e}_2-{\bf e}_1, \ldots, {\bf e}_n-{\bf e}_1$ of the affine span of the $T_k$,
we see that the coefficient of ${\bf e}_i-{\bf e}_1$ is nonnegative if $i\in 1\cup X_k$ and nonpositive if $i\in 1\cup Y_k$, for $k=1,2$.
Therefore this coefficient is zero unless $i\in X_1\cap X_2$ or $i\in Y_1\cap Y_2$, and this implies $p\in T_3$.
\end{proof}

The following result is an immediate corollary of Lemma \ref{lattice}:
\begin{cor}
Define $T_Y=u\del_X+{\bf e}_B+t\nab_Y$. For every face $F$ of the mixed subdivision, if $F$ is contained in any top degree face, then the set of $Y$ such that $F$ is contained in $T_Y$ is an interval in the boolean lattice.
\end{cor}

We now have all the ingredients we need to prove the main result of this section, restated here:
\begin{thmcoeffs}
Take the regular mixed subdivision $\mathcal F$ of $u\Delta+P(M)+t\nabla$ as described above. 
The unsigned coefficient $|[x^iy^j]Q'_M|$ counts the cells $F+G+H$ of~$\mathcal F$ where $i=dim(F)$, $j=dim(H)$,
and $G$ is a vertex of $P(M)$ and there exists no cell $F+G'+H$ where $G'\supsetneq G$. 
\end{thmcoeffs}

\begin{proof}

First, we must show that all the lattice points of $P(M)+u\Delta+t\nabla$ lie in a top degree face.
Recall that $\pi$ is the projection map from $\mathbb{R}^{n+1}\rightarrow\mathbb{R}^n$.

\begin{sublemma}
\label{claim5}
 Any $\operatorname{\pi}(x)\in (t\nabla+P(M)+u\Delta)\cap\mathbb{Z}^n$ on $\mathfrak{F}$ is of the form $(-{\bf e}_{i_1},\beta_1)+\cdots +(-{\bf e}_{i_t},\beta_t)+({\bf e}_B,0)+({\bf e}_{j_1},\alpha_1)+\cdots +({\bf e}_{j_u},\alpha_u)$.
\end{sublemma}

\begin{subproof}[Proof of Claim~\ref{claim5}]
Lemma~\ref{712} provides an expression for $x$ of the form
\[x=-{\bf e}_{i_1}-\cdots -{\bf e}_{i_t}+{\bf e}_B+{\bf e}_{j_1}+\cdots +{\bf e}_{j_u}.\]
Taking arbitrary preimages under $\pi$ gives the requisite expression except that the sum may be incorrect in the last coordinate. 
To obtain an expression where the last coordinate is correct, we will now rewrite this to show that $x$ in fact lies on $\mathfrak {F}$: this is equivalent to showing that there exists a partition $X\sqcup Y=[n]\setminus 1$ such that every $i$ is in $1\cup Y$, every $j$ is in $1\cup X$, and the algorithm of Theorem \ref{lem:B}, given $X$ and $Y$, yields $B$. This is because, as we know the height function used to lift the top-degree faces, finding this $X$ and $Y$ will give a top-degree face containing $\operatorname{\pi}(x)$, and we would then have the correct last coordinate.

The postconditions of Algorithm~\ref{alg:1} require that 
\begin{enumerate}
\item if there exists an element $d\notin B$ such that $d<e\in B$ and $(B-e)\cup d\in\mathcal{B}$ then $e\in X$, and
\item if there exists an element $e\in B$ such that $e<f\notin B$ and $(B-e)\cup f\in\mathcal{B}$ then $f\in Y$.
\end{enumerate}
Choose any $X$ and any $Y$. Given these, we will construct $X'$ and $Y'$ such that $X'\sqcup Y'=[n]\setminus 1$.
\begin{itemize}
\item Suppose $e\in B$ and $e=i_k$, and there exists $d<e$ with $d\in B$ such that $(B-e)\cup d\in\mathcal{B}$. In the expression for $x$, replace $-{\bf e}_e+{\bf e}_B$ with $-{\bf e}_d+{\bf e}_{(B-e)\cup d}$. Add $d$ to $Y'$.
\item Suppose $f\in B$ and $f=j_l$, and there exists $e<f$ with $e\in B$ such that $(B-e)\cup f\in\mathcal{B}$. In the expression for $x$, replace ${\bf e}_B+{\bf e}_f$ with ${\bf e}_{(B-e)\cup f}+{\bf e}_e$. Add $e$ to $X'$.
\item If we have an element $i\in X\cap Y$, in the expression for $x$ replace $-{\bf e}_i+{\bf e}_i$ with $-{\bf e}_1+{\bf e}_1$. Remove $i$ from both $X'$ and $Y'$.
\end{itemize}
The above three operations always replace a term $\pm {\bf e}_a$ with a smaller term. As we have a finite ground set, there is a finite amount of such operations, and so this construction must terminate with a $X',Y'$ which fits the criteria. At this point, the expression we have for $x$ will be that required by the claim.
\end{subproof}

Continuing the proof of the theorem, we now form a poset $P$ where the elements are the top degree faces and all nonempty intersections of sets of these, ordered by containment. 
This poset is a subposet of the face lattice of the $(|E|-1)$-dimensional cube
whose vertices correspond to the top degree faces. 
Proposition~\ref{prop:top degree} shows that every lattice point of 
$u\Delta+P(M)+t\nabla$ lies in at least one face in~$P$.
The total number of lattice points is given by inclusion-exclusion on
the function on~$P$ assigning to each element of~$P$ the number of lattice points
in that face. Let $[\cdot]$ denote the number of lattice points of the corresponding face. So we have that 
\begin{equation}
\label{eqn}
Q_M(t,u)=\sum_{i,j}c_{ij}\binom{u}{j}\binom{t}{i}=\sum_{k\geq 1}(-1)^k\sum_{\substack{S\subseteq \textrm{atoms}(P)\\ |S|=k}}\left[\bigwedge S\right]
\end{equation}
$$\qquad\qquad=\sum_{x\in P}\mu (0,x)[x]$$
where $\mu$ is the M\"{o}bius function. Now, as the face poset of the cubical complex $\mathcal{C}$ is Eulerian, we have that $\mu (x,y)=(-1)^{r(y)-r(x)}$. This means that

\begin{equation}
\label{eqn2}
\sum_{k\geq 1}(-1)^k\sum_{\substack{S\subseteq \textrm{atoms}(P)\\ |S|=k}}\left[\bigwedge S\right]=\sum_{E \ \textrm{face \ of} \ \mathcal{C}} (-1)^{\textrm{codim} E}\binom{t+i}{i}\binom{u+j}{j}
\end{equation}
 where $E$ is the product of an $i$-dimensional face of $\del$ with a $j$-dimensional face of $\nab$, that is, $E$ corresponds to a face of type $tF+G+uH$, where $G$ is a basis of $P(M)$.

$Q'_M$ expands as
$$Q'_M(x,y)=\sum_{i,j}c_{ij}(x-1)^i(y-1)^j=\sum_{i,j,k,l}c_{ij}\binom{i}{k}x^k(-1)^{i-k}\binom{j}{l}y^l(-1)^{j-l}$$
in which the coefficient of $x^ky^l$ is $\sum_{i,j}c_{ij}\binom{i}{k}(-1)^{i-k}\binom{j}{l}(-1)^{j-l}$. To compare this to the count in the lattice, we need to expand $\binom{t}{i}$ (and $\binom{u}{j}$) in the basis of $\binom{t+i}{i}$ (and $\binom{u+j}{j}$). This gives that $$\binom{t}{i}=\sum_{k=0}^i (-1)^{i-k}\binom{i}{k}\binom{t+k}{k},$$ as proven below.

\begin{sublemma}
\label{claim6}
For any positive integers $i,t$,
$$\binom{t}{i}=\sum_{k=0}^i (-1)^{i-k}\binom{i}{k}\binom{t+k}{k}.$$
\end{sublemma}

\begin{subproof}[Proof of Claim~\ref{claim6}]
The Vandermonde identity gives that $$\binom{t+i}{i}=\sum_{k=0}^i \binom{t}{k}\binom{i}{i-k}=\sum_{k=0}^i\binom{t}{k}\binom{i}{k}.$$ We will use the binomial inversion theorem to get $\binom{t}{i}$. Rewriting the above identity in this language, we have that $$f_i=\sum_{k=0}^ig_k\binom{i}{k},$$ where $f_i=\binom{t+i}{i}$ and $g_k=\binom{t}{k}$. Then, $$g_i=\sum_{k=0}^i(-1)^{i+k}f_k\binom{i}{k}.$$ Substituting in values again gives that $$\binom{t}{i}=\sum_{k=0}^i(-1)^{i+k}\binom{t+k}{k}\binom{i}{k},$$ as in the statement of the claim.
\end{subproof}
Substitute this into Equation \ref{eqn} to get
$$\sum_{k\geq 1}(-1)^k\sum_{\substack{S\subseteq \textrm{atoms}(P)\\ |S|=k}}\left[\bigwedge S\right]=\sum_{i,j,k,l}c_{ij}(-1)^{i-k}\binom{i}{k}\binom{t+k}{k}(-1)^{j-l}\binom{j}{l}\binom{u+l}{l}$$
$$\qquad\qquad = \sum_{k,l}[x^ky^l]Q'_M(x,y)\binom{t+k}{k}\binom{u+l}{l}.$$
Comparing this to Equation \ref{eqn2} proves Theorem \ref{coeffs}.
\end{proof}




The above proof immediately yields as a corollary that the signs of the coefficients of $Q'_M(x,y)$ are alternating. 

\begin{cor}\label{cor:alternating}
$(-1)^{|E|-1}Q'_M(-x,-y)$ has nonnegative coefficients in $x$ and~$y$.
\end{cor}

This is not dissimilar to the Tutte polynomial, whose coefficients are all nonnegative.  
The coefficients of $Q'_M$, up to sign, have the combinatorial interpretation
of counting elements of~$P$ of form $u\Delta_X + {\bf e}_B + t\nabla_Y$
by the cardinalities of $X\setminus\{1\}$ and~$Y\setminus\{1\}$.
In particular the top degree faces are counted by the collection
of coefficients of~$Q'_M$ of top degree (hence the name),
and the degree $|E|-1$ terms of $Q'_M$ are always $(x+y)^{|E|-1}$.




The appearance of basis activities in Lemma~\ref{lem:B} reveals that
$P$ is intimately related to a familiar object in matroid theory,
the \emph{Dawson partition} \cite{dawson}. 
Give the lexicographic order to the power set $\mathcal P(E)$.
A partition of $\mathcal P(E)$ into intervals $[S_1,T_1],\ldots,[S_p,T_p]$ with indices such that $S_1<\ldots<S_p$  is a Dawson partition if and only if $T_1<\ldots < T_p$. 
Every matroid gives rise to a Dawson partition in which these intervals are $[B\setminus\mathrm{Int}(B),B\cup\mathrm{Ext}(B)]$ for all $B\in\mathcal B_M$.

\begin{prop}\label{prop:Dawson}
Let $[S_1,T_1],\ldots,[S_p,T_p]$ be the Dawson partition of~$M$.
The poset $P$ is a disjoint union of face posets of cubes $C_1,\ldots,C_p$
where the vertices of $C_i$ are the top-degree faces $u\Delta_X+{\bf e}_B+t\nabla_Y$ 
such that $X\in[S_i,T_i]$.
\end{prop}

The description of the cubes comes from Lemma \ref{lattice}. Note that the element $1$ is both internally and externally active with respect to every basis, due to it being the smallest element in the ordering. So, even though $1$ is in both $X$ and $Y$, it is in $T_i-S_i$ for all $i$.


\section{Polymatroids}
\label{sec:polymatroids}
In this section we investigate the invariants $Q_M$ and $Q'_M$ when $M=(E,r)$ is a polymatroid. 
Many familiar matroid operations have polymatroid counterparts, 
and we describe the behaviour of $Q'_M$ under these operations.
We see that it retains versions of several formulae true of the Tutte polynomial.

For instance, there is a polymatroid analogue of the direct sum of matroids:
given two polymatroids $M_1=(E_1,r_1),M_2=(E_2,r_2)$ with disjoint ground sets,
their \emph{direct sum} $M=(E,r)$ has ground set $E = E_1\sqcup E_2$
and rank function $r(S) = r_1(S\cap E_1)+r_2(S\cap E_2)$.
This definition extends the usual direct sum of matroids.

\begin{prop}\label{prop:direct sum}
Let $M_1\oplus M_2$ be the direct sum of two polymatroids $M_1$ and $M_2$.
Then 
\[Q'_{M_1\oplus M_2}(x,y)=(x+y-1)\cdot Q'_{M_1}(x,y)\cdot Q'_{M_2}(x,y).\]
\end{prop}

\begin{proof}
%
We will need to use distinguished notation for our standard simplices according to the matroid under consideration.
So we write $\del_i = \conv\{\mathbf e_j : j\in E_i\}$ for $i=1,2$
and reserve the unadorned name $\del$ for $\conv\{\mathbf e_j : j\in E\}$.
Define $\nab_i$ and $\nab$ similarly.

The basic relationship between $M=M_1\oplus M_2$ and its summands in terms of our lattice point counts is the following equality:
\begin{equation}\label{eq:ziggurat}
\sum_{k=0}^{\min\{t,u\}} Q_M(t-k,u-k)=
\sum_{t_1=0}^{t}\sum_{u_1=0}^{u}Q_{M_1}(t_1,u_1)\cdot Q_{M_2}(t-t_1,u-u_1).
\end{equation}

The right hand side counts tuples $(t_1,u_1,q_1,q_2)$ 
where $q_1\in (P(M_1)+u_1\del_1+t_1\nab_1)\cap\mathbb Z^{E_1}$
and $q_2\in (P(M_2)+(u-u_1)\del_2+(t-t_1)\nab_2)\cap\mathbb Z^{E_2}$.
Because the coordinate inclusions of $\del_1$ and $\del_2$ are subsets of $\del$,
and similarly for $\nab$, the concatenation $(q_1,q_2)\in\mathbb Z^E$
is a lattice point of $P(M)+u\del+t\nab$.
The left hand side counts pairs $(k,q)$ where 
$q$ is a lattice point of $P(M)+(u-k)\del+(t-k)\nab$;
this polyhedron is a subset of $P(M)+u\del+t\nab$.
To prove equation~\eqref{eq:ziggurat} we will show that each $q=(q_1,q_2)\in\mathbb Z^E$
occurs with the same number of values of $k$ on the left as values of $(t_1,u_1)$ on the right.

If $q\in(P(M)+u\del+t\nab)\cap\mathbb Z^E$ then there is some maximal integer $K$ such that
$q\in(P(M)+(u-K)\del+(t-K)\nab)\cap\mathbb Z^E$, and then $q$ is counted just $K+1$ times on the left hand side,
namely for $k=0,1,\ldots,K$.
This $K$ is what we called the number of \emph{cancellations} in $q$ when proving Theorem~\ref{thm:T to Q}.
Choose an expression
\[q=e_B+{\bf e}_{x_1}+\cdots+{\bf e}_{x_{u-K}}-{\bf e}_{y_1}-\cdots -{\bf e}_{y_{t-K}},\]
where $B$ is a basis of $M$ and the $x_i$ and $y_i$ are elements of $E$.
Suppose the $x_i$ and $y_i$ are ordered such that 
$x_i\in E_1$ if and only if $i\leq t'$ and $y_i\in E_1$ if and only if $i\leq u'$.
Then we have
\begin{align*}
q_1 &= e_{B\cap E_1}+{\bf e}_{x_1}+\cdots+{\bf e}_{x_{u'}}-{\bf e}_{y_1}-\cdots -{\bf e}_{y_{t'}},\\
q_2 &= e_{B\cap E_2}+{\bf e}_{x_{u'+1}}+\cdots+{\bf e}_{x_{u-K}}-{\bf e}_{y_{t'+1}}-\cdots -{\bf e}_{y_{t-K}},
\end{align*}
and both of these expressions also have the minimal number of ${\bf e}_{x_i}$ and ${\bf e}_{y_i}$ summands,
or else our expression for $q$ would not have done.
Thus $q_1$ is in $P(M_1)+u'\del+t'\nab$ but not $P(M_1)+(u'-1)\del+(t'-1)\nab$,
and $q_2$ is in $P(M_1)+(u-u'-K)\del+(t-t'-K)\nab$ but not $P(M_1)+(u-u'-K-1)\del+(t-t'-K-1)\nab$.
So the possibilities for $t_1$ and $u_1$ on the right hand side of \eqref{eq:ziggurat}
are those that arrange $t_1\geq t'$ and $t-t_1\geq t-t'-K$, and the corresponding equations for the $u$ variables,
together with $t_1-u_1=t'-u'$.  
There are exactly $K+1$ solutions here as well, namely $t_1 = t', t'+1, \ldots, t'+K$.

Using equation~\eqref{eq:ziggurat} within a generating function for $t$ and~$u$, we have that
\begin{align*}
&\mathrel{\phantom=}
\sum_{t,u}\sum_{k=0}^{\min\{t,u\}}Q_{M_1\oplus M_2}(t-k,u-k)v^uw^t\\ &=\sum_{t,u}\sum_{t_1=0}^{t}\sum_{u_1=0}^{u}Q_{M_1}(t_1,u_1)v^{u_1}w^{t_1} \cdot Q_{M_2}(t-t_1,u-u_1)v^{u-u_1}w^{t-t_1} \\
& = \sum_{t_1,u_1}Q_{M_1}(t_1,u_1)v^{u_1}w^{t_1}\\
&\hspace{15mm}\cdot\sum_{t-t_1,u-u_1}Q_{M_2}(t-t_1,u-u_1)v^{u-u_1}w^{t-t_1} \\
& = \left(\sum_{t,u}Q_{M_1}(t,u)v^{u}w^{t}\right)\left(\sum_{t,u}Q_{M_2}(t,u)v^{u}w^{t}\right).
\end{align*}
The left-hand side can also be simplified as
\begin{align*}
&\mathrel{\phantom=}
\sum_{t,u}\sum_{k=0}^{\min\{t,u\}}Q_{M_1\oplus M_2}(t-k,u-k)v^uw^t\\
&=\sum_{k\geq 0}\sum_{t,u\geq k}Q_{M_1\oplus M_2}(t-k,u-k)v^{u-k}w^{t-k}(vw)^k\\
&= \sum_{k\geq 0}\sum_{t,u}Q_{M_1\oplus M_2}(t,u)v^uw^t(vw)^k\\
&= \dfrac{1}{1-vw}\cdot\sum_{t,u}Q_{M_1\oplus M_2}(t,u)v^uw^t
\end{align*}
and thus we have that 
$$\dfrac{1}{1-vw}\cdot\sum_{t,u}Q_{M_1\oplus M_2}(t,u)v^uw^t= \left(\sum_{t,u}Q_{M_1}(t,u)v^{u_1}w^{t_1}\right)\left(\sum_{t,u}Q_{M_2}(t,u)v^{u_1}w^{t_1}\right).$$

The generating functions above allow us easily to change basis from $Q_M$ to $Q'_M$:
\begin{align*}
\sum_{t,u}Q_M(t,u)v^uw^t &= \sum_{t,u}\sum_{i,j}c_{ij}\binom{u}{j}\binom{t}{i}v^uw^t\\
&= \sum_{i,j}c_{ij}\dfrac{v^j}{(1-v)^{j+1}}\cdot\dfrac{w^i}{(1-w)^{i+1}}\\
&= \dfrac{1}{(1-v)(1-w)}\cdot Q'_M\left(\dfrac{w}{1-w}+1,\dfrac{v}{1-v}+1\right)
\end{align*}
where the last line follows from the definition of $Q'_M$. Hence we have 
$$Q'_{M_1\oplus M_2}\left(\dfrac{1}{1-w},\dfrac{1}{1-v}\right)=\dfrac{1-vw}{(1-v)(1-w)}Q'_{M_1}\left(\dfrac{1}{1-w},\dfrac{1}{1-v}\right)Q'_{M_2}\left(\dfrac{1}{1-w},\dfrac{1}{1-v}\right)$$
and substituting $w=\dfrac{x-1}{x}$ and $v=\dfrac{y-1}{y}$ gives the result.
\end{proof}

\begin{remark}
It follows from Proposition~\ref{prop:direct sum}
that the rescaled matroid invariant $(x+y-1)\cdot Q'_M(x,y)$ is exactly multiplicative under direct sum.
Recasting Theorem~\ref{thm:T to Q} in terms of this rescaled invariant also eliminates a denominator.
And, by Proposition~\ref{prop:Dawson}, its coefficients can be interpreted as
counting intervals in the Boolean lattice $\mathcal P(E)$ contained in a single part of a Dawson partition
according to the ranks of their minimum and maximum,
with no need to accord a special role to one element.
\end{remark}

Moving on to the next polymatroid operation, 
it is apparent from the symmetry of Theorem~\ref{lpg} under switching $x$ and~$y$
that $Q'_M$, like the Tutte polynomial, exchanges its two variables under matroid duality.
The best analogue of duality for polymatroids requires a parameter $s$
greater than or equal to the rank of any singleton; then if $M=(E,r)$ is a polymatroid,
its \emph{$s$-dual} is the polymatroid $M^*=(E,r^*)$ with 
\[r^*(S) = r(E) + s|E\setminus S| - r(E\setminus S).\]
The 1-dual of a matroid is its usual dual.

\begin{prop}\label{prop:duality}
For any polymatroid $M=(E,r)$ and any $s$-dual $M^*$ of~$M$, 
$Q'_{M^*}(x,y)=Q'_M(y,x)$.
\end{prop}

\begin{proof}
Let $\phi:\mathbb R^E\to\mathbb R^E$ be the involution that subtracts every coordinate from~$s$. 
Definition~\ref{dfn:polytope} implies that
$\phi$ is a bijection which takes elements of $P(M)$ to elements of $P(M^*)$;
it also clearly preserves the property of being a lattice point. 
This gives that 
\begin{align*}
\#(P(M)+u\Delta+t(-\Delta))&=\#(\phi(P(M))+u\phi(\Delta)+t\phi(-\Delta))\\
& =\#(P(M^*)+u(\nabla+1_E)+t(\Delta+1_E)\\
& =\#(P(M^*)+(t-u)1_E+u\nab+t\Delta)\\
&= \#(P(M^*)+u\nab+t\Delta)
\end{align*}
where the last line is true due to the polytope being a translation of the one in the line above. The statement follows.
\end{proof}

Given a subdivision $P_1,\ldots,P_n$ of a polytope $P$, a \emph{valuation} is a function $f$ such that 
\[f(P)=\sum\limits_{P_i}f(P_i)-\sum\limits_{P_i,P_j}f(P_i\cap P_j)+\ldots+(-1)^{n-1} f(P_1\cap\cdots \cap P_n).\] 
The number of lattice points in a polytope is a valuation: if $f(P)$ is this counting function,
then it is easily checked that each lattice point counted in $f(P)$ also contributes exactly one to the sum on the right hand side.
Thus the invariant $Q'_M$ is a valuation as well. 
\begin{prop}
Let $\mathcal F$ be a polyhedral complex whose total space is a polymatroid base polytope
$P(M)$, and each of whose faces $F$ is a polymatroid base polytope $P(M(F))$.  Then
\[Q'_M(x,y) = \sum_{\mbox{\scriptsize $F$ a face of $\mathcal F$}}
(-1)^{\dim(P(M))-\dim F} Q'_{M(F)}(x,y).\]
\end{prop}

For example, if $M$ is a matroid and we relax a circuit-hyperplane, we get the following result:

\begin{cor}
Take a matroid $M=(E,r)$ and let $C\subset E$ be a circuit-hyperplane of $M$. Let $M'$ be the matroid formed by relaxing $C$. Then $Q'_M(x,y)=Q'_{M'}(x,y)-x^{n-r(M)-1}y^{r(M)-1}$.
\end{cor}

Now consider deletion and contraction in polymatroids.
We have that $P(M\backslash a)=\{p\in P(M) \ | \ p_a=\underline{k}\}$, where $\underline{k}$ is the minimum value $p_a$ takes (this will be $0$ unless $a$ is a coloop), and that $P(M/a)=\{p\in P(M) \ | \ p_a=\overline{k}\}$, where $\overline{k}$ is the maximum value $p_a$ takes. When $M$ is a matroid, these two sets partition $P(M)$. However, when $M$ is a polymatroid, we can have points in $P(M)$ where $\underline{k}<p_a<\overline{k}$. Let $N_k:=\{p\in P(M) \ | \ p_a=k\}$, and let $P(N_k)$  be the polytope consisting of the convex hull of such points. Now we have that $P(M\backslash a)$, $P(M/a)$, and the collection of $P(N_k)$ for $k\in\{\underline{k}+1,\ldots,\overline{k}-1\}$ partition $P(M)$. We will refer to each of these parts, when they exist, as an \emph{$a$-slice} of $P(M)$. When we do not include the deletion and contraction slices, we can talk about \emph{(strictly) interior} slices.

\begin{thm}
\label{delcont}
Let $M=(E,r)$ be a polymatroid and take $a\in E(M)$. Then 
$$Q'_M(x,y)=(x-1)Q'_{M\backslash a}(x,y)+(y-1)Q'_{M/a}(x,y)+\sum_{N} Q'_N(x,y).$$
where $N$ ranges over $a$-slices of $P(M)$. 
\end{thm}

Note that when $M$ is a matroid, the statement simplifies to the formulae given in Lemma \ref{recurrence}: if $a$ is neither a loop nor coloop, then the $a$-slices are  $P(M\backslash a)$ and $P(M/a)$, so  
\begin{align*}
Q'_M(x,y)&=(x-1)Q'_{M\backslash a}(x,y)+(y-1)Q'_{M/a}(x,y)+\sum_{N_k} Q'_{N_k}(x,y)\\
&=(x-1)Q'_{M\backslash a}(x,y)+(y-1)Q'_{M/a}(x,y)+Q'_{M\backslash a}(x,y)+Q'_{M/a}(x,y)\\
&=xQ'_{M\backslash a}(x,y)+yQ'_{M/a}(x,y).
\end{align*}
When $a$ is a loop or coloop, $M\backslash a=M/a$, and we have only one $a$-slice: $P(M\backslash a)=P(M/a)$. So we get that 
\begin{align*}
Q'_M(x,y)&=(x-1)Q'_{M\backslash a}(x,y)+(y-1)Q'_{M/a}(x,y)+\sum_{N_k} Q'_{N_k}(x,y)\\
&= (x+y-1)Q'_{M/a}(x,y)
\end{align*}
as in Lemma \ref{recurrence}.

Also note that this result gives another proof of Theorem \ref{thm:T to Q} as a corollary.

\begin{proof}
In this proof, we make constant use of Lemma~\ref{712} in order to express lattice points as sums of lattice points.
Let $M$ be a polymatroid. If the rank function of $M$ is a matroid rank function summed with a function of the form $S\mapsto\sum\limits_{i\in S} c_i$, then $P(M)$ will be a translate of a matroid polytope, and the same argument as above will hold. Assume now that this is not the case. This means that for any $a\in E(M)$, there will be at least one $a$-slice of $P(M)$, $P(N_k)$, which is not equal to $P(M/a)$ or $P(M\backslash a)$. 
\begin{sublemma}
\label{claim1}
Define $R$ to be the polytope $\{q\in P(M)+u\Delta_E+t\nabla_E \ | \ q_a=k\}$, and define $S$ to be $P(N_k)+u\Delta_{E-a}+t\nabla_{E-a}$. If $R$ intersects the set of lattice points of $P(M)$, then $R=S$.
\end{sublemma}

\begin{subproof}[Proof of Claim~\ref{claim1}]
It is clear that the lattice points of $S$ are contained in $R$. Take a point in $R$, $q_1=p_1+{\bf e}_{i_1}+\cdots+{\bf e}_{i_u}-{\bf e}_{j_1}-\cdots -{\bf e}_{j_t}$. We will show that we can write this as a point contained in $S$, $q_2=p_2+{\bf e}_{m_1}+\cdots +{\bf e}_{m_{u}}-{\bf e}_{n_1}-\cdots -{\bf e}_{n_t}$, where no ${m_i}$ or ${n_j}$ can be equal to $a$. 

If $(p_1)_a=k$, then we simply choose $p_2$ to be $p_1$ and choose $m_k=i_k$, $n_k=j_k$ for all $k\in\{1,\ldots,t\}$, with one possible change: if we have ${i_k}={j_l}=a$ in $q_1$, in $q_2$ replace ${m_k}$ and ${n_l}$ with $b$, where $b$ is any other element in $E(M)$. Note ${\bf e}_a$ must always appear paired in this way, such that $(q_1)_a=k$, and so this change does not affect the coordinate values of $q_2$.

If $(p_1)_a\neq k$, we first must rewrite the expression for $q_1$. 
By the base exchange property for polymatroids (\cite[Theorem 4.1]{hibi}), given $p_1$ and any point $p_3\in P(M)$, if $(p_1)_i>(p_3)_i$ there exists $l$ such that $(p_1)_l< (p_3)_l$ and $p_1-{\bf e}_i+{\bf e}_l\in P(M)$. 
Let $(p_1)_a=k+\lambda$, where $\lambda>0$. Then, by repeatedly applying the exchange property, we get that $p_1-\lambda {\bf e}_a+{\bf e}_{l_1}+\cdots+{\bf e}_{l_\lambda}\in P(M)$. Then we can find $q_2$ by setting $p_2=p_1-\lambda {\bf e}_a+{\bf e}_{l_1}+\cdots+{\bf e}_{l_\lambda}$, so $$q_2=p_2+\lambda {\bf e}_a-{\bf e}_{l_1}-\cdots-{\bf e}_{l_\lambda}+{\bf e}_{i_1}+\ldots+{\bf e}_{i_u}-{\bf e}_{j_1}-\cdots -{\bf e}_{j_t}=q_1.$$ Note that as $(q_1)_a=k$ and $(p_2)_a=k$, there must be $\lambda$ $-{\bf e}_{j_k}$ terms equal to $-{\bf e}_a$, so
$$q_2=p_2-{\bf e}_{l_1}-\cdots-{\bf e}_{l_\lambda}+{\bf e}_{i_1}+\ldots+{\bf e}_{i_u}-{\bf e}_{j_1}-\cdots -{\bf e}_{j_{t-\lambda}}$$
which is of the correct form, completing the proof of Claim~\ref{claim1}.
\end{subproof}

\begin{sublemma}
\label{6.12}
Let $N_i$ be a strictly interior slice of $P(M)$. Then $P(M)+t\Delta_E+t\nabla_{E}=(P(M/a)+u\del_E+t\nab_{E-a}) \ \sqcup \ \bigsqcup\limits_i (P(N_i)+t\del_{E-a}+t\nab_{E-a}) \ \sqcup \ (P(M\backslash a)+t\del_{E-a}+t\nab_E)$.
\end{sublemma}

\begin{subproof}[Proof of Claim~\ref{6.12}]
Take $P(M)+u\del_E+t\nab_E$ and split it into a collection of polytopes according to the value of $q_a$ for all points $q\in P(M)+u\del_E+t\nab_E$. The disjoint union of the lattice points of these parts clearly will give back those of the original polytope. By the previous result, if one of these parts intersects $P(M)$ we can write it as $P(N_k)+u\del_{E-a}+t\nab_{E-a}$. Otherwise, we must be able to write the part as $P(M/a)+(u-\lambda)\del_{E-a}+\lambda {\bf e}_a+t\nab_{E-a}$, where $\lambda>\ol{k}$, or as $P(M\backslash a)+t\del_{E-a}-\mu {\bf e}_a+(t-\mu)\nab_{E-a}$, where $\mu>\underline{k}$. 

We will show that 
\begin{equation}
\label{raise}
\bigsqcup_\lambda P(M/a)+(u-\lambda)\del_{E-a}+\lambda {\bf e}_a+t\nab_{E-a}=P(M/a)+u\del_E+t\nab_{E-a}.
\end{equation}
It is clear that the sets of lattice points of the summands are pairwise disjoint as the $a$-coordinates in each set must be different. It is also clear that the lattice points contained in the polytope on the left hand side are contained in that of the right hand side. Take a point $q_1=p_1+{\bf e}_{i_1}+\cdots+{\bf e}_{i_u}-{\bf e}_{j_1}-\cdots -{\bf e}_{j_t}$ contained in $P(M/a)+u\del_E+t\nab_{E-a}$. Let $(q_1)_a=\ol{k}+\mu$, where $\mu>0$. We need to write $q_1$ as $p_2+{\bf e}_{m_1}+\cdots +{\bf e}_{m_{u-\lambda}}+\lambda {\bf e}_a-{\bf e}_{n_1}-\cdots -{\bf e}_{n_t}$, a lattice point contained in one of the summands on the left hand side. Choose $\mu=\lambda$, $p_2=p_1$, $\{j_\alpha\}=\{n_\alpha\}$, and $\{{i_\beta} \ | \ i_\beta\neq a\}=\{{m_\beta}\}$ and the equality follows.

The same arguments show that
\begin{equation}
\label{lower}
\bigsqcup_\mu P(M\backslash a)+t\del_{E-a}-\mu {\bf e}_a+(t-\mu)\nab_{E-a}=P(M\backslash a)+u\del_{E-a}+t\nab_E
\end{equation}
and the claim follows.
\end{subproof}

\begin{sublemma}
\label{claim2}
We have that
$$\#(P(M/a)+u\del_E+t\nab_{E-a})=\sum_{j=0}^u\#(P(M/a)+j\del_{E-a}+t\nab_{E-a})$$
and
$$\#(P(M\backslash a)+u\del_{E-a}+t\nab_{E})=\sum_{j=0}^u\#(P(M\backslash a)+u\del_{E-a}+i\nab_{E-a}).$$
\end{sublemma}

\begin{subproof}[Proof of Claim~\ref{claim2}]
Take the cardinalities of both sides of Equations \ref{raise} and \ref{lower}.
\end{subproof}

Continuing the proof of the theorem, we now that have 
\begin{equation}
\label{almost}
Q_M(t,u)=\sum_{N_k} Q_{N_k}(t,u)+\sum_{j=0}^u Q_{M/a}(t,j)+\sum_{i=0}^t Q_{M\backslash a}(i,u)
\end{equation}
where $k\in\{\underline{k}+1,\ldots,\ol{k}-1\}$, that is, $N_k$ is always a strictly interior slice of $P(M)$.

We now work out how the change of basis from $Q$ to $Q'$ transforms the sums in Equation \ref{almost}. Take a term in $Q_{M/a}$, $c_{ik}\binom{j}{k}\binom{t}{i}$. We have that
\begin{align*}
\sum_{j=0}^uc_{ik}\binom{j}{k}\binom{t}{i} &=c_{ik}\binom{u+1}{k+1}\binom{t}{i}\\
&= c_{ik}\binom{t}{i}\left(\binom{u}{k}+\binom{u}{k+1}\right).
\end{align*}
Now apply the change of basis to get
\begin{align*}
c_{ik}(x-1)^i((y-1)^k+(y-1)^{k+1}) & = c_{ik}(x-1)^i\left((y-1)^k(1+y-1)\right)\\
&=c_{ik}(x_1)^i(y-1)^ky.
\end{align*}
Thus 
$$\sum_{j=0}^u Q_{M/a}(t,j)=yQ'_{M/a}(t,u)$$
and similarly, $$\sum_{i=0}^t Q_{M\backslash a}(i,u)=xQ'_{M\backslash a}(t,u).$$

Finally, putting this together with Claim~\ref{6.12} and Equations \eqref{raise}, \eqref{lower} gives:
\begin{align*}
Q'_M(t,u)&=xQ'_{M/a}(t,u)+yQ'_{M/a}(t,u)+\sum_{\textrm{interior} \ N_k} Q'_{N_k}(t,u) \\
&= (x-1)Q'_{M/a}(t,u)+(y-1)Q'_{M/a}(t,u)+\sum_{N_k} Q'_{N_k}(t,u).
\end{align*}
This completes the proof of Theorem \ref{delcont}.
\end{proof}

Unfortunately, when $M$ is a polymatroid, there is no analogue to Corollary~\ref{cor:alternating}:
the coefficients of~$Q'_M$ do not have sign independent of~$M$, and thus
there can be no straightforward enumerative interpretation of the coefficients. 
This is a consequence of the failure of Theorem~\ref{coeffs} for polymatroids.
Here is an example to illustrate this.
\begin{ex}
The left of Figure~\ref{fig:2} displays the subdivision $\mathcal F$ for the sum of Example~\ref{ex:1}.
\begin{figure}
\begin{center}
\includegraphics[scale=0.666667]{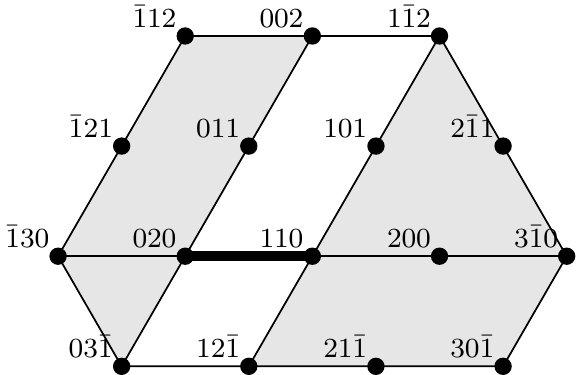}
\hspace{0.5in}
\includegraphics[scale=0.666667]{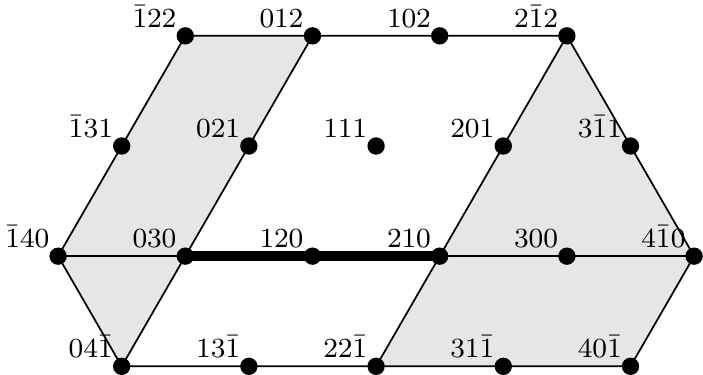}
\end{center}
\caption{At left, the regular subdivision $\mathcal F$ associated
to the Minkowski sum of Example~\ref{ex:1}, with 
$P(M)$ bolded and the top degree faces shaded in grey.
At right, the regular subdivision for a related polymatroid,
still with $(t,u)=(2,1)$.}\label{fig:2}
\end{figure}
We see that the four grey top degree faces contain all the lattice points between them,
and the poset $P$ contains two other faces which are pairwise intersections thereof,
the horizontal segment on the left with $(X,Y) = (1,12)$ and the one on the right with $(X,Y) = (12,1)$.
These are indeed enumerated, up to the alternation of sign,
by the polynomial $Q'_M(x,y) = x^2 + 2xy + y^2 - x - y$ found earlier.

By contrast, the right of the figure displays $\mathcal F$
for the polymatroid $M_2$ obtained by doubling the rank function of~$M$.
The corresponding polynomial is $Q'_{M_2}(x,y) = x^2 + 2xy + y^2 - 1$,
in which the signs are not alternating, dashing hopes of a similar enumerative interpretation.
In the figure we see that there are lattice points not on any grey face.
\end{ex}

\section{K\'alm\'an's activities}
\label{sec:kalman}
One motivation for the particular Minkowski sum we have employed in our definition
is that it provides a polyhedral translation of K\'alm\'an's construction of activities in a polymatroid.
This section explains the connection.
We first recall the definitions of K\'alm\'an's univariate activity invariants of polymatroids. 
These polynomials do not depend on the order on~$E$ that was used to define them \cite[Theorem~5.4]{kalman}. 

\begin{citedfn}[{\cite{kalman}}]
\label{activity}
Define the \emph{internal polynomial} and \emph{external polynomial} of a polymatroid $M=(E,r)$ by
\begin{displaymath}I_M(\xi)=\sum_{x\in \mathcal B_M\cap\mathbb{Z}^E} \xi^{\ol{{\iota}}(x)} \quad \mathrm{and} \quad X_M(\eta)=\sum_{x\in \mathcal B_M\cap\mathbb{Z}^E} \eta^{\ol{\varepsilon}(x)}.\end{displaymath}
\end{citedfn}

\begin{lemma}
\label{old}Let $P$ be a polymatroid polytope. At every lattice point $f\in P$, attach the scaled simplex \[f+t\operatorname{conv}(\{-e_i \ | \ i \ \textrm{is internally active in} \ f \ \textrm{or} \ i\notin f \}).\] This operation partitions $P+t\nabla $
into a collection of translates of faces of $t\nabla $,
with the simplex attached at $f$ having codimension $\ol{\iota}(f)$
within $P$.
\end{lemma}

\begin{figure}[ht]
\begin{center}
\includegraphics{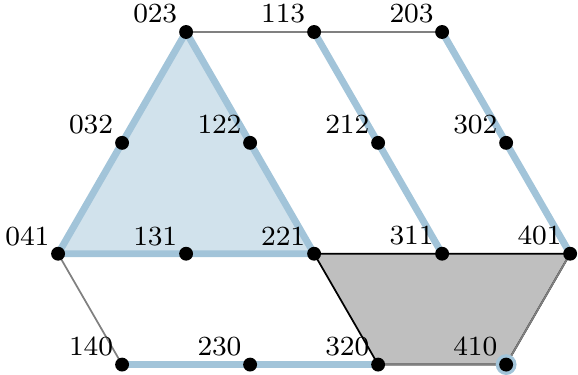}
\end{center}
\caption{An example of the partition of Lemma~\ref{old}.
A translate of $P$ is shown, each of its lattice points lying in the corresponding simplex of the partition.}\label{activitypartition}
\end{figure}

Figure \ref{activitypartition} shows a case of this operation, to illustrate why we speak of ``attaching'' a simplex. 
Our polymatroid $P(M)$ is coloured grey, and the polytope drawn is $P(M)+2\del+\nab$.
In the picture we translate $P(M)$ by a multiple of ${\bf e}_1$, namely $1{\bf e}_1$, in order to allow both polygons to reside in the same plane.
Coordinate labels are written without parentheses and commas. The blue areas are faces of the scaled simplices $2\del_f$, and can be seen to be a partition of the lattice points.
(Don't overlook the blue dot at 410!)

\begin{proof}
We will first show that our simplices covers all the lattice points of $P+t\nab$. 
Let $g\in P+t\nab$ be a lattice point. We will find $f$ such that $g\in t\Delta_f$.

Let $g_t=g$. For $i\in\{0,\ldots,t-1\}$, define
\begin{displaymath}
   g_i = \left\{
     \begin{array}{ll}
       g_{i+1}+{\bf e}_1 & \textrm{if} \ g_{i+1}+{\bf e}_1 \in P+(t-i)\nab\\
       g_{i+1}+{\bf e}_2 & \textrm{if} \ g_{i+1}+{\bf e}_1 \notin P+(t-i)\nab,  g_{i+1}+{\bf e}_2 \in P+(t-i)\nab\\
       \qquad \vdots & \\
       g_{i+1}+{\bf e}_n & \textrm{if} \ g_{i+1}+{\bf e}_{h} \notin P+(t-i)\nab, \ \forall j\in [n], \ g_{i+1}+{\bf e}_n \in P+(t-i)\nab
            \end{array}
   \right.
\end{displaymath} 

In other words, at each iteration $i$, we are adding an element ${\bf e}_j$ which is internally active with respect to $g_{i+1}$. We cannot replace ${\bf e}_j$ with ${\bf e}_i$ where $i<j$ and remain inside $P+t\nab$. Let ${\bf e}_{j_i}$ be the element added in iteration $t$. We get that $$g=g_t=g_0-{\bf e}_{j_1}-\ldots -{\bf e}_{j_t}\in P+t\nab.$$ Note that if we added ${\bf e}_i$ at some stage $g_s$ of the iteration, and ${\bf e}_j$ at stage $g_{s-1}$, then $j\geq i$. Thus if we take a tuple ${\bf e}_{k}$ such that $(k_1,\ldots,k_t)<(j_1,\ldots,j_t)$ with respect to the lexicographic ordering, then $g_0-\sum \limits_{t}{\bf e}_{j_t}+\sum\limits_{t} {\bf e}_{k_t}\notin P$, so each ${\bf e}_j$ is internally active. Thus $g_0=f$ and the ${\bf e}_j$ found define a simplex $\Delta_f$ such that $g\in t\Delta_f$.

Now we will show that this operation gives disjoint sets. We have that $\{t\del_f\}$ covers $P+t\nab$, and that $\{(t-1)\del_f\}$ partitions $P+(t-1)\nab$. Thus in order to show that $\{t\del_f\}$ is in fact a partition of the lattice points of $P$, it suffices to prove that if $g_t\in t\Delta_f$, then $g_{t-1}\in(t-1)\del_{f}$. Say that $f=g_t+{\bf e}_{i_1}+\cdots+{\bf e}_{i_t}$. This means that each element ${\bf e}_{i_k}$ is internally active at $f$ for all $k\in\{0,\ldots,t\}$. Now, for a contradiction, let $g_0=f'\neq f$, so that $g_{t-1}=g_t+{\bf e}_i\in (t-1)\Delta_{f'}$. Apply the same iterative process as before to get
$$ 
\begin{array}{ll}
f'=g_0& =  g_{t-1}+{\bf e}_{i_1}+\cdots+{\bf e}_{i_{t-1}}\\
& =  g_t+{\bf e}_{i}+{\bf e}_{i_1}+\cdots+{\bf e}_{i_{t-1}}\\
& =  f-{\bf e}_{i_t}+{\bf e}_i.
\end{array}$$
Thus ${\bf e}_{i_t}$ was internally inactive at $f$, contradicting our construction of $t\Delta_f$.
\end{proof}

The following is a direct consequence of this lemma,
and its exterior analogue which arises from replacing $\iota$ and $\nabla$
with $\varepsilon$ and $\Delta$.  

\begin{thm}\label{thm:activity}
Let $M$ be a polymatroid.
Then $I_M(\xi) = \xi\cdot Q'_M(\xi,1)$
and $X_M(\eta) = \eta\cdot Q'_M(1,\eta)$.
\end{thm}

It is this result which first motivated
the particular change of basis we have made from $Q_M$ to~$Q'_M$,
since an $i$-dimensional face of $t\Delta$ 
has $\displaystyle\binom{t+i}i = \sum_{k=0}^i\binom ik\binom ti$ lattice points.

The bivariate enumerator of internal and external activities for
polymatroids is not order-independent, and so we do not have that $Q'_M=\sum_{x\in\mathcal{B}_M\cap\mathbb{Z}^{|E|}}\xi^{\iota (x)}\eta^{\varepsilon (x)}$. 

\begin{ex}\label{bivaractivity}
Take the polymatroid with bases $\{(0,2,1),(1,1,1),(1,2,0),(2,1,0),(2,0,1)\}$.
Using the natural ordering on $[3]$, we have that $\sum\limits_{x\in\mathcal{B}_M\cap\mathbb{Z}^{|E|}}\xi^{\iota (x)}\eta^{\varepsilon (x)}=\xi^3\eta+2\xi^2\eta^2+\xi\eta^2+\xi\eta^3.$ If we instead use the ordering $2<3<1$, the enumerator is $\xi^3\eta^2+\xi^2\eta^2+\xi^2\eta+\xi\eta^2+\xi\eta^3.$
\begin{figure}[ht]
\begin{center}
\includegraphics[scale=1.25]{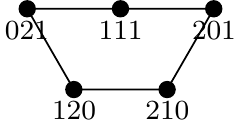}
\end{center}
\caption{The polymatroid of Example \ref{bivaractivity}.}
\end{figure}
\end{ex}

\begin{question}
Section~10 of~\cite{kalman} is dedicated to the behaviour of K\'alm\'an's activity invariants
in trinities in the sense of Tutte.  
One can obtain six hypergraphs from a properly three-coloured triangulation of the sphere
by deleting one colour class and regarding the second and third as vertices and hyperedges of a hypergraph.
K\'alm\'an considered the relationships between values of his invariants on these six hypergraphs.
With the proof of his main conjecture in~\cite{kalman2}, 
we know that besides the internal and external invariants $I_G$ and $X_G$ of a hypergraph $G$ with plane embedding,
there exists a third invariant $Y_G$
such that the values of the three invariants are permuted by the action of the symmetric group $S_3$ on the colour classes in the natural way.

Our work has cast $I_M$ and $X_M$ as univariate evaluations of a bivariate polynomial $Q'_M$,
and the content of Proposition~\ref{prop:duality} is that polymatroid duality,
i.e.\ exchanging the deleted and vertex colour classes, exchanges the two variables of $Q'_M$.
Is there a good trivariate polynomial of a three-coloured triangulation of the sphere which similarly encapsulates the above observations on trinities?
Permuting the three colours should permute its three variables, 
and three of its univariate evaluations should be $I_G$, $X_G$, and $Y_G$.
We would of course be even happier if $Q'_G$ were among its bivariate evaluations.
\end{question}

\bibliographystyle{plain} 
\bibliography{library}
\label{sec:biblio}

\end{document}